\documentclass{amsart} 

\usepackage{pstricks}
\include{pst-plot}
\usepackage{amsmath,amsthm}     
\usepackage{amssymb}            
\usepackage{euscript}           
\usepackage{enumerate,calc,graphicx,lscape,color}
\definecolor{red}{rgb}{1,0,0}
\usepackage[matrix,arrow,curve,frame]{xy}    

\xymatrixcolsep{1.9pc}                          
\xymatrixrowsep{1.9pc}
\newdir{ >}{{}*!/-5pt/\dir{>}}                  

\newtheorem{thm}[subsection]{Theorem}
\newtheorem{defn}[subsection]{Definition}
\newtheorem{prop}[subsection]{Proposition}

\newtheorem{cor}[subsection]{Corollary}
\newtheorem{lemma}[subsection]{Lemma}

\theoremstyle{definition}  
\newtheorem{ex}[subsection]{Example}

\newtheorem{remark}[subsection]{Remark}

  {\end{list}}

\newcommand{\dfn}{\textbf} 

\newcommand{\mdfn}[1]{\dfn{\mathversion{bold}#1}} 

\newcommand{\Smash}             {\wedge}

\newcommand{\Wedge}             {\vee}
\DeclareRobustCommand{\bigWedge}{\bigvee}
\newcommand{\iso}               {\cong}

\newcommand{\cat}{\EuScript}    
\newcommand{\cC}{{\cat C}}

\newcommand{\cM}{{\cat M}}

\newcommand{\Spectra}{{\cat Spectra}}
\newcommand{\Motspectra}{{\cat MotSpectra}}

\DeclareMathOperator{\Rad}{Rad}

\newcommand{\sSet}{s{\cat Set}}

\newcommand{\sPre}{sPre}

\newcommand{\Ho}{\text{Ho}\,}

\newcommand{\field}[1]  {\mathbb #1} 
\newcommand{\A}         {\field A}
\newcommand{\R}         {\field R}
\newcommand{\Q}         {\field Q}

\newcommand{\HH}        {\field H}
\newcommand{\OO}        {\field O}
\newcommand{\Z}         {\field Z}
\newcommand{\C}         {\field C}

\newcommand{\F}         {\field F}
\renewcommand{\P}         {\field P}

\DeclareMathOperator*{\colim}{colim}

\DeclareMathOperator*{\hocolim}{hocolim}

\DeclareMathOperator{\Spec}{Spec}

\DeclareMathOperator{\adj}{adj}

\DeclareMathOperator{\id}{id}

\DeclareMathOperator{\inv}{inv}

\newcommand{\ra}{\rightarrow}                   
\newcommand{\lra}{\longrightarrow}              
\newcommand{\la}{\leftarrow}                    
\newcommand{\lla}{\longleftarrow}               
\newcommand{\llra}[1]{\stackrel{#1}{\lra}}      
\newcommand{\llla}[1]{\stackrel{#1}{\lla}}      

\newcommand{\cof}{\rightarrowtail}              
\newcommand{\bcof}{\leftarrowtail}              

\newcommand{\inc}{\hookrightarrow}              

\newcommand{\blank}{-}                          
\newcommand{\und}{\underline}

\newcommand{\ovcat}{\downarrow}

\newcommand{\he}{\simeq}

\newcommand{\Sm}{Sm}

\newcommand{\rea}[1]{|{#1}|}             
\newcommand{\map}{\rightarrow}

\newcommand{\ceck}[1]{\Cech(#1)}         
\newcommand{\oceck}[1]{\Cech^{o}(#1)}    
\newcommand{\oreal}[1]{\rea{\oceck{U}}}  
\newcommand{\creal}[1]{\rea{\ceck{U}}}   

\newcommand{\Cech}{\check{C}}

\renewcommand{\top}{\textrm{top}}

\numberwithin{equation}{subsection}

\newenvironment{myequation}
  {\addtocounter{subsection}{1}\begin{eqnarray}}
  {\end{eqnarray}$\!\!$}

\renewcommand{\S}{\mathbb{S}}

\newcommand{\SL}{SL}

\newcommand{\qf}[1]{{\langle{#1}\rangle}}
\newcommand{\GW}{GW}
\newcommand{\lab}[1]{{\qf{#1}}}
\DeclareMathOperator{\Cof}{Cof}

\newgray{gridline}{0.9}
\newcommand{\parrow}{$\xymatrixcolsep{1.2pc}\xymatrix{{} \ar@{=>}[r] &
{}}$}
\newcommand{\lparrow}{$\xymatrixcolsep{1.2pc}\xymatrix{{}  & {}
\ar@{=>}[l]}$}
\newcommand{\swparrow}{$\xymatrixcolsep{1pc}\xymatrix{{}  & {}
\ar@{=>}[dl] \\
{}& {}}$}

\begin{document}

\title{Motivic Hopf elements and relations}

\author{Daniel Dugger}
\author{Daniel C. Isaksen}

\address{Department of Mathematics\\ University of Oregon\\ Eugene, OR
97403}

\address{Department of Mathematics\\ Wayne State University\\
Detroit, MI 48202}

\email{ddugger@math.uoregon.edu}

\email{isaksen@math.wayne.edu}



\begin{abstract}
We use Cayley-Dickson algebras to produce Hopf elements $\eta$, $\nu$,
and $\sigma$
in the motivic stable homotopy groups of spheres, and we prove the
relations $\eta\nu=0$ and $\nu\sigma=0$ by geometric arguments.  Along
the way we develop several basic facts about the motivic stable
homotopy ring.
\end{abstract}

\maketitle

\tableofcontents

\section{Introduction}

The work of Morel and Voevodsky \cite{MV,V} has
shown how to construct from the category $\Sm/k$ of smooth schemes
over a commutative ring $k$ 
a corresponding motivic stable homotopy category.  This
comes to us as the homotopy category of a model category of
motivic symmetric spectra \cite{Ho,J}.  Among the motivic spectra are certain
``spheres'' $S^{p,q}$ for all $p,q\in \Z$, so that for any motivic
spectrum $X$ one obtains the bi-graded stable homotopy groups
$\pi_{*,*}(X)=\oplus_{p,q} [S^{p,q},X]$.  This paper deals with the
construction of some elements and relations in the motivic stable
homotopy ring $\pi_{*,*}(S)$, where $S$ is the sphere spectrum.

\medskip

Classically there are two ways of trying to compute stable homotopy
groups.  First there was the hands-on approach
of Hopf, Toda, Whitehead, and others, where
one constructs explicit elements and explicit relations.  
Of course
this is difficult and painstaking.  
Later on, Serre's thesis, and its ultimate
realization in the Adams spectral sequence, greatly reduced the
difficulties in calculation---but at the expense of computing the
homotopy groups of the completions $S^{\Smash}_p$, not $S$ itself.
Fortunately these things are closely related: $\pi_j(S^{\Smash}_p)$ is
just the $p$-completion of $\pi_j(S)$.  

In the motivic setting the analog of the Adams spectral sequence was
explored in \cite{DI} over an
algebraically closed field.  
For other fields the computations are much more
challenging, even for fields like $\R$ and $\Q$.  The motivic
Adams-Novikov spectral sequence over algebraically closed fields was
considered in \cite{HKO}. Recent work of Ormsby-\O stvaer studies
related issues over the $p$-adic fields $\Q_p$ \cite{OO}.

In the present paper our goal is to explore a little of the
``hands-on'' approach of Hopf, Toda, and Whitehead to motivic homotopy groups.  
That is to say, our goal is
to construct very explicit elements of these groups and to demonstrate
some relations that they satisfy.
Most of our results work over an arbitrary base;
equivalently, they work over the universal base $\Z$.  But in practice
it is often useful to  assume that the base $k$ is either a field 
or the integers $\Z$.
Occasionally we will restrict to the case of a field, for purposes of
exposition.

\medskip

In comparison to Adams spectral sequence computations, the hands-on
constructions considered in this paper are very grueling.  The ratio
of effort versus payoff is fairly large.
For this reason we give a few remarks about the motivation for
pursuing this line of inquiry.  

A drawback of the Adams spectral sequence methods is that the spectral
sequences converge only to the homotopy groups of a suitable completion
$S^{\Smash}_{H}$, based on the choice of a prime $p$.  Unlike the
classical case, appropriate finiteness theorems are not available and
{\it a priori\/} there can be a significant difference between the
motivic homotopy groups themselves, their completions, and the
homotopy groups of the completed sphere spectrum.  
The motivic Adams spectral sequences compute highly
interesting objects, regardless of their exact relationship to the
motivic stable homotopy groups.  For example, one can use these
motivic spectral sequences  to learn about classical and equivariant stable
homotopy theory, even without identifying the motivic completion
$S^{\Smash}_{H}$ precisely.  But while these techniques lead to
interesting results, one still cannot help but wonder about the nature
of the ``true'' motivic homotopy groups.

One important aspect of the motivic stable homotopy groups is that
they act as operations on every (generalized) motivic cohomology
theory.  And although it is tautological, it is useful to keep in mind
that the motivic stable homotopy ring $\pi_{*,*}(S)$ {\it equals\/}
the ring of universal motivic cohomology operations.  Studying this
ring thereby gives us potential tools relevant to algebraic $K$-theory
and algebraic cobordism, for example.  In contrast, studying the
motivic homotopy groups of the Adams completions only give tools
relevant to completed versions of $K$-theory and cobordism.

Another drawback of the motivic Adams spectral sequence approach is
that it only applies in situations where one has very detailed
information about the structure of the motivic Steenrod algebra.  This
information is available when working over an essentially smooth
scheme over a field whose characteristic is different from the chosen
prime $p$
\cite{HKO2}.  However, it is not clear whether these results can be
extended to schemes that are not defined over a field, such as $\Spec
\Z$.  Moreover, when $p$ equals the characteristic of the base field
the structure of the motivic Steenrod algebra is likely to be more
complicated.

\medskip

\subsection{Background}
It follows from Morel's connectivity theorem \cite{morel-stable} 
that the motivic stable homotopy groups of spheres vanish in a certain
range: $\pi_{p,q}(S)=0$ for
$p<q$.  The group $\oplus_p \pi_{p,p}(S)$ is called the ``$0$-line'',
and was completely determined by Morel.  It will be useful to briefly
review this.

Recall that $S^{1,1}=(\A^1-0)$.  For each $a\in k^{\times}$ let $\rho_a\colon
S^{0,0}\ra S^{1,1}$ be the map that sends the basepoint to $1$ and the
nonbasepoint to $a$.  This gives a homotopy element $\rho_a$ in
$\pi_{-1,-1}(S)$. 
We write $\rho$ for $\rho_{-1}$ because, 
as we will see, this element plays a special role.
 
Furthermore, performing the Hopf construction (cf. Appendix~\ref{se:Hopf}) on the multiplication
map $(\A^1-0)\times (\A^1-0)\ra (\A^1-0)$ gives a map $\eta\colon S^{3,2}\ra
S^{1,1}$, and therefore a corresponding element $\eta$ in
$\pi_{1,1}(S)$.  
Finally, let
$\epsilon\colon S^{1,1}\Smash S^{1,1}\ra S^{1,1}\Smash S^{1,1}$ be
the twist map.  It represents an element in $\pi_{0,0}(S)$.

Morel's theorem \cite{M1} is the following.

\begin{thm}[Morel \cite{M2}]
\label{th:Morel}
Let $k$ be a perfect field whose characteristic is not 2.
The ring $\oplus_n \pi_{n,n}(S)$ is the free associative algebra
generated by the elements $\eta$ and $\rho_a$ (for all $a$ in $k^\times$)
subject to the following relations:
\begin{enumerate}[(i)]
\item $\eta\rho_a=\rho_a\eta$ for all $a$ in $k^\times$;
\item $\rho_a\cdot \rho_{1-a}=0$ for all $a\in k-\{0,1\}$;
\item $\eta^2\rho+2\eta=0$;
\item $\rho_{ab}=\rho_a+\rho_b+\eta\rho_a\rho_b$, for all $a,b\in
k^\times$;
\item $\rho_1=0$.
\end{enumerate}
Additionally, one has $\epsilon=-1-\rho\eta$.  
\end{thm}

The relations in Theorem \ref{th:Morel} have a number of algebraic consequences,
some of which are interesting for their own sakes. 
For example, it follows through a lengthy  chain of manipulations
that $\rho_a \rho_b = \epsilon \rho_b \rho_a$
\cite[Lemma 2.7(3)]{morel-alg-top}.
This is a special case of a more general formula 
from Proposition \ref{pr:commute}
concerning commutativity in the motivic stable homotopy ring.

There is a map of symmetric monoidal categories 
\[ \Ho(\Spectra)\ra
\Ho(\Motspectra)
\]  
that sends a spectrum to the corresponding
``constant presheaf''.  The technical details are unimportant here,
only that this gives a map $\pi_n(S)\ra \pi_{n,0}(S)$ from the
classical stable homotopy groups to their motivic analogs.   For an
element $\theta\in \pi_n(S)$ let us write $\theta_{top}$ for its image
in $\pi_{n,0}(S)$.  So, for example, we have the elements
$\eta_{top}$ in $\pi_{1,0}(S)$, $\nu_{top}$ in $\pi_{3,0}(S)$, and
$\sigma_{top}$ in $\pi_{7,0}(S)$.  

At this point our exposition has reached the limit of what is
available in the literature.  
No complete computation has been made of any stable motivic homotopy
group $\pi_{p,q}(S)$ for $p>q$.  (For some computations of unstable
homotopy groups, though, see \cite{AF}).

\subsection{Statements of results}

Using a version of Cayley-Dickson algebras we construct elements
$\nu$ in $\pi_{3,2}(S)$ and $\sigma$ in $\pi_{7,4}(S)$.  Taken together
with $\eta$ in $\pi_{1,1}(S)$ we call these the \dfn{motivic Hopf
elements}.    There is also a zeroth Hopf element: classically this is
$2$ in $\pi_0(S)$, but in the motivic context it turns out to be better
to take this to be $1-\epsilon$ in $\pi_{0,0}(S)$ (we will 
see why momentarily).

Morel shows in \cite{M2} that the relation $\epsilon\eta=\eta$ follows from
commutativity of the multiplication map $\mu\colon (\A^1-0)\times
(\A^1-0)\ra \A^1-0$.  We offer the general philosophy that properties of
the higher Cayley-Dickson algebras should give rise to relations
amongst the Hopf elements.  Teasing out such relations from the
properties of the algebras is a tricky business, though.  
In this paper we prove the following generalization of Morel's result:

\begin{thm}
\label{th:null-Hopf}
 $(1-\epsilon)\eta=\eta\nu=\nu\sigma=0$.
\end{thm}

The three relations fit an evident pattern: the product of two
successive Hopf elements is zero.  We call this the ``null-Hopf
relation''.  Notice that the pattern of three equations 
 provides some motivation for
regarding $1-\epsilon$ as the zeroth Hopf map.
The following corollary is worth recording:

\begin{cor}
\label{cor:epsilon-nu}
$\epsilon \nu=-\nu$.
\end{cor}

\begin{proof}
$\epsilon \nu=(-1-\rho \eta)\nu=-\nu$, since $\eta\nu=0$.
\end{proof}

As a long-term goal it would be nice to completely determine the
subalgebra of $\pi_{*,*}(S)$ generated by the motivic Hopf elements,
the elements $\rho_a$, and the image of $\pi_*(S)\ra \pi_{*,0}(S)$.  
These constitute the part of $\pi_{*,*}(S)$ that is ``easy to write
down''.  Completion of this goal seems far away, however.

There are other evident geometric sources for maps between spheres.
One class of examples are the $n$th power maps $P_n\colon (\A^1-0)\ra
(\A^1-0)$.  These give elements of $\pi_{0,0}(S)$, and we completely
identify these elements in Theorem \ref{th:diag-power} below.  
Another group of
examples are the
diagonal maps $\Delta_{p,q}\colon S^{p,q}\ra S^{p,q}\Smash S^{p,q}$.  
In classical topology these are all null-homotopic, and most of them
are null motivically as well.  There is, however, an exception when $p=q$:

\begin{thm}\mbox{}\par
\label{th:diag-power}
\begin{enumerate}[(a)]
\item For $n\geq 0$ the diagonal map $\Delta\colon S^{n,n}\ra
S^{n,n}\Smash S^{n,n}$ represents $\rho^{n}$ in $\pi_{-n,-n}(S)$.  
\item For $p>q\geq 0$ the diagonal map $\Delta\colon S^{p,q} \ra
S^{p,q}\Smash S^{p,q}$ is null homotopic.
\item For $n$ in $\Z$, the $n$th power map $P_n\colon (\A^1-0) \ra (\A^1-0)$
represents  
\[ 
\begin{cases}
\frac{n}{2}(1-\epsilon) & \text{if $n$ is even},\\
1+\frac{n-1}{2}(1-\epsilon) &\text{if $n$ is odd.}
\end{cases}
\]
\end{enumerate}
\end{thm}

The various facts in the above theorem are useful in a variety of
circumstances, but there is a specific reason for including them in the
present paper: all three parts play a role in the proof of the
null-Hopf relation from Theorem~\ref{th:null-Hopf}.

\subsection{Next Steps}

This paper does not exhaust the possibilities of the ``hands-on" approach 
to motivic stable homotopy groups over $\Spec \Z$.
An obvious next step is to consider generalizations of the classical relation
$12\nu = \eta^3$ in $\pi_3(S)$.  

This formula as written cannot possibly hold motivically, since the left side
belongs to $\pi_{3,2}(S)$ while the right side belongs to $\pi_{3,3}(S)$.
An obvious substitution is to ask whether
$12\nu$ equals $\eta^2 \eta_{\top}$ in $\pi_{3,2}(S)$.
One might speculate that the $12\nu$
should be replaced by $6(1-\epsilon)\nu$, but these expressions are already known
to be equal in $\pi_{3,2}(S)$ by 
Corollary \ref{cor:epsilon-nu}.

Another possible extension concerns Toda brackets.  
Classically, the Toda bracket $\langle \eta, 2, \nu^2 \rangle$ in $\pi_8(S)$
contains an element called ``$\epsilon$" that is a multiplicative generator
for the stable homotopy ring.  (Beware that this bracket has indeterminacy
generated by $\eta \sigma$.)
Motivically, we can form the Toda bracket
$\langle \eta, 1-\epsilon, \nu^2 \rangle$ in $\pi_{8,5}(S)$ and obtain
a motivic generalization over $\Spec \Z$. 
(There is a notational conflict here because $\epsilon$ is used in the motivic
context for the twist map in $\pi_{0,0}(S)$.)

There is much more to say about Toda brackets in this context, but we will
leave the details for future work.

\subsection{Organization of the paper}
There is a certain amount of technical machinery needed for the paper,
and this has all been deposited into three appendices.  The body of
the paper has been written assuming knowledge of these appendices, but
the most efficient way to read the paper might be to 
first ignore them, referring back only as needed for technical
details.  Appendix A deals with stable splittings of smash
products inside of Cartesian products.  Appendix B deals with joins and
also certain issues of ``canonical isomorphisms'' in homotopy theory.
Finally, Appendix C treats the Hopf construction and related
issues; there is a key idea of ``melding'' two pairings together, and
a recondite formula for the Hopf construction of such a melding
(Proposition~\ref{pr:H(meld)}).  
This formula is perhaps the 
most important technical element in our proof of the null-Hopf relation.  

Concerning the main body of the paper, Section 2 reviews basic facts
about the motivic stable homotopy category and the ring
$\pi_{*,*}(S)$.  An important issue here is the precise
definition of what it means for a map in the stable homotopy category
to represent an element of $\pi_{*,*}(S)$, and also formulas for
dealing with the ``motivic signs'' that inevitably arise in
calculations.  

Section 3 deals with diagonal maps and power maps, and there we prove
Theorem~\ref{th:diag-power}.  Section 4 reviews the necessary material
about Cayley-Dickson algebras and defines the motivic Hopf elements
$\eta$, $\nu$, and $\sigma$.  Finally, in Section 5 we prove the
null-Hopf relation of Theorem~\ref{th:null-Hopf}. 

\subsection{Notation}
We remark that the symbols $\chi$ and $p$, when applied to maps, have
a special meaning in this paper.  Maps called $p$ are always the
projection from a Cartesian to a smash product, and maps called $\chi$
are certain stable splittings for these projections.  See
Appendix~\ref{se:stable-split} for details.

\subsection{Acknowledgements}
The first author was supported by NSF grant DMS-0905888.  The second
author was supported by NSF grant DMS-1202213.


\section{Preliminaries}

This section describes certain foundational issues and conventions regarding
the motivic stable homotopy category and the motivic stable homotopy ring
$\pi_{*,*}(S)$.  

\subsection{Basic setup}
Fix a commutative ring $k$ (in practice this will usually be $\Z$ or a
field).
Let $\Sm/k$ denote the category of smooth schemes
over $\Spec k$.  The category of  {\it motivic spaces\/} is the category of
simplicial presheaves $\sPre(\Sm/k)$.  This category carries various
Quillen-equivalent model structures that represent unstable
$\A^1$-homotopy theory, but for the purposes of this paper we will
mostly use the injective model structure developed in \cite{MV}.
It is very convenient that all objects are cofibrant in this
structure.  We will usually shorten ``motivic spaces'' to just
``spaces'' for the rest of the paper.  

Most of the paper actually restricts to the setting of pointed motivic
spaces.  This is the associated  model category $\sPre(\Sm/k)_*=(*\!\ovcat\!
\sPre(\Sm/k))$ of motivic spaces under $*$.

As explained in \cite{J}, one can stabilize the category of pointed motivic
spaces to form a model catgory of motivic symmetric spectra.  We write
$\Motspectra$ for this category.  Our aim in this paper is to work in
the homotopy category $\Ho(\Motspectra)$ as much as possible, and this
is where all of our theorems take place.  As is usual in homotopy
theory, however, a certain amount of work necessarily has to take
place at the model category level.

It is useful to be able to compare the motivic homotopy category to
the classical homotopy category of topological spaces, and there are a
couple of ways to do this.  The ``constant presheaf'' functor is the
left adjoint in a Quillen pair $\sSet \ra 
\sPre(\Sm/k)$ (when we write Quillen pairs we draw an arrow in the
direction of the left adjoint).  This stabilizes to a similar Quillen pair between
symmetric spectra categories
\[ \Spectra \llra{c} \Motspectra.\]

Alternatively, if the base ring $k$ is embedded in $\C$ then we can
`realize' our motivic spaces as ordinary topological spaces, and
likewise realize motivic spectra as ordinary spectra.  Unfortunately
this doesn't work well at the model category level if we use the
injective model structure, as we do not get Quillen pairs.  For these
comparison purposes it is more convenient to use the flasque model
structure of \cite{I}.  We will not need the details in the present
paper, only the fact that this can be done; we occasionally refer to
topological realization in a passing comment.

\subsection{Spheres and the ring \mdfn{$\pi_{*,*}(S)$}}
We begin with the two objects $S^{1,0}$ and $S^{1,1}$ in
$\Ho(\Motspectra)$.   Here $S^{1,0}=\Sigma^\infty S^1$, 
where $S^1$ is the ``simplicial circle", i.e.,
the constant presheaf with value $S^1$.
Likewise, $S^{1,1}$ is the suspension spectrum of 
the representable presheaf $(\A^1-0)$,
which has basepoint given by the rational point
$1$ in $(\A^1-0)$.

Let us fix motivic spectra 
 $S^{-1,0}$ and $S^{-1,-1}$
together with isomorphisms (in the homotopy category) 
$a_1\colon S^{-1,0}\Smash S^{1,0}\ra S^{0,0}$ and
$a_2\colon S^{-1,-1}\Smash S^{1,1}\ra S^{0,0}$.  There is some choice
involved in these isomorphisms, as they can be varied by an arbitrary
self-homotopy equivalence of the spectrum $S^{0,0}$.  For $a_1$ it is
convenient to fix the corresponding isomorphism $a\colon S^{-1}\Smash S^1\ra
S^0$ in $\Spectra$ and then let $a_1$ be the image of $a$ under the derived
functor of
$c$.  For $a_2$ it is perhaps best to fix a choice once and for all
over $\Spec \Z$, and 
to insist that the 
topological realization of $a_2$ is $a$; this is not strictly
necessary, however.

For each integer $n$, define 
\[ 
S^{n,0}=\begin{cases} 
(S^{1,0})^{\Smash(n)} & \text{if $n\geq 0$,}\\
(S^{-1,0})^{\Smash(-n)} & \text{if $n<0$,}
\end{cases}
\]
\[
S^{n,n}=\begin{cases} 
(S^{1,1})^{\Smash(n)} & \text{if $n\geq 0$,}\\
(S^{-1,-1})^{\Smash(-n)} & \text{if $n<0$.}
\end{cases}
\]
Finally, for integers $p$ and $q$, define
\[ S^{p,q}=(S^{1,0})^{\Smash(p-q)}\Smash (S^{1,1})^{\Smash(q)}.
\]

We will need the following important result from \cite{D}: for any
$(p_1,q_1),\ldots,(p_n,q_n)$ in $\Z^2$ and
$(p_1',q_1'),\ldots,(p_k',q_k')$ in $\Z^2$ 
such that $\sum_i p_i=\sum_i p'_i$ and $\sum_i q_i=\sum_i q'_i$,
there is a uniquely--distinguished  ``canonical isomorphism''
\[ \phi\colon S^{p_1,q_1}\Smash \cdots \Smash S^{p_n,q_n} \ra
S^{p'_1,q'_1}\Smash \cdots \Smash S^{p'_k,q'_k}
\]
in the homotopy category of motivic spectra.
These canonical isomorphisms have the properties that:
\begin{itemize}
\item If $\phi$ and $\phi'$ are canonical then so are
$\phi\Smash\phi'$ and $\phi^{-1}$;
\item Identity maps are all canonical, as are the maps $a_1$ and $a_2$;
\item The unit maps $S^{p,q}\Smash S^{0,0}\iso S^{p,q}$ and
$S^{0,0}\Smash S^{p,q}\iso S^{p,q}$ are canonical;
\item Any composition of canonical maps is canonical.
\end{itemize}
See \cite[Remark 1.9]{D} for a complete discussion.
We will always denote these canonical morphisms by the symbol $\phi$.
(Note: In this paper we systematically suppress all associativity
isomorphisms; but if we were not suppressing them, they would also be
canonical).

Define $\pi_{p,q}(S)=[S^{p,q},S^{0,0}]$ and write $\pi_{*,*}(S)$ for
$\oplus_{p,q} \pi_{p,q}(S)$.  If $f\in \pi_{a,b}(S)$ and $g\in
\pi_{c,d}(S)$ define $f\cdot g$ to be the composite
\[ S^{a+c,b+d}\llra{\phi} S^{a,b}\Smash S^{c,d} \llra{f\Smash g}
S^{0,0}\Smash S^{0,0} \iso S^{0,0}.
\]  
By \cite[Proposition 6.1(a)]{D} this product makes $\pi_{*,*}(S)$ into an associative
and unital ring, where the subring $\pi_{0,0}(S)$ is central.  

\subsection{Representing elements of \mdfn{$\pi_{*,*}(S)$}}
\label{se:signs}

Let $f\colon S^{a,b}\ra S^{p,q}$. 
We write $|f|$ for $(a-p,b-q)$, i.e., the bidegree of the motivic 
stable homotopy element that $f$ will represent.
There are two ways to obtain an
element of $\pi_{a-p,b-q}(S)$ from $f$, which we will denote $[f]_l$ and
$[f]_r$.  Let $[f]_l$ be the composite
\[ S^{a-p,b-q}\llra{\phi} S^{a,b}\Smash S^{-p,-q}\llra{f\Smash \id}
S^{p,q}\Smash S^{-p,-q}\llra{\phi} S^{0,0}\]
and let $[f]_r$ be the composite
\[
 S^{a-p,b-q}\llra{\phi} S^{-p,-q}\Smash S^{a,b}\llra{\id\Smash f}
S^{-p,-q}\Smash S^{p,q}\llra{\phi} S^{0,0}.
\]
It is proven in \cite[Section 6.2]{D} that $[gf]_r=[g]_r\cdot [f]_r$, whereas
$[gf]_l=[f]_l\cdot [g]_l$.  In this paper we will never use $[f]_l$,
and so we will just write $[f]=[f]_r$.  

For each $a,b,p,q\in \Z$ let $t_{(a,b),(p,q)}$ denote the composition
\[ S^{a+p,b+q}\llra{\phi} S^{a,b}\Smash S^{p,q} \llra{t} S^{p,q}\Smash
S^{a,b} \llra{\phi} S^{a+p,b+q}
\]
where $t$ is the twist isomorphism for the smash product.  Write
$\tau_{(a,b),(p,q)}=[t_{(a,b),(p,q)}]\in \pi_{0,0}(S)$.  It is easy
to see that  $\tau_{1,0}=-1$, as this formula holds in 
$\Ho(\Spectra)$ and one just pushes it into $\Ho(\Motspectra)$ via
the functor $c$.  
Let 
$\epsilon=\tau_{1,1}$.  The following formula is then a special case of
 \cite[Proposition 6.6]{D}:
\begin{myequation}
\label{eq:tau}
\tau_{(a,b),(p,q)}=(-1)^{(a-b)\cdot(p-q)}\cdot \epsilon^{b\cdot q}.
\end{myequation}
Note that $\tau\colon \Z^2\times \Z^2 \ra \pi_{0,0}(S)^\times$ is bilinear.

These elements $\tau_{(a,b),(p,q)}$ arise in various formulas related to
commutativity of the smash product.  For example,
the following is from \cite[Proposition 1.18]{D}:

\begin{prop}[Graded-commutativity]
\label{pr:commute}
Let $f\in \pi_{a,b}(S)$ and $g\in \pi_{c,d}(S)$.  Then
\[ fg=gf\cdot \tau_{(a,b),(c,d)}=gf\cdot (-1)^{(a-b)(c-d)}\cdot \epsilon^{bd}.
\]
\end{prop}

\begin{remark}
\label{re:invcoh}
If $f\colon S^{a,b}\ra S^{p,q}$ then we may consider the two maps $f\Smash
\id_{r,s}$ and $\id_{r,s}\Smash f$.  All three of these maps represent
elements in $\pi_{*,*}(S)$, but not necessarily the same ones.  
The following two facts are proven in
\cite[Proposition 6.11]{D}:
\begin{enumerate}[(i)]
\item $[\id_{r,s}\Smash f]=[f]$
\item $[f\Smash
\id_{r,s}]=\tau_{|f|,(r,s)}\cdot [f]=\tau_{(a-p,b-q),(r,s)}[f]=(-1)^{(a-p-b+q)\cdot(r-s)}
\epsilon^{(b-q)s}[f]$.  
\end{enumerate}
A useful  special case says  that if $f\colon S^{a,b}\ra S^{a,b}$ then
$[f]=[\id_{r,s}\Smash f]=[f\Smash \id_{r,s}]$.  

If $g\colon S^{r,s}\ra S^{t,u}$ then combining (i) and (ii) we obtain
\begin{enumerate}[(i)]
\addtocounter{enumi}{2}
\item $[f\Smash g]=[(f\Smash \id_{t,u})\circ (\id_{a,b}\Smash g)]=[f\Smash
\id_{t,u}]\cdot [\id_{a,b}\Smash g]=[f]\cdot [g] \cdot \tau_{|f|,(t,u)}$.
\end{enumerate}
\end{remark}

\subsection{Homotopy spheres}

We will often study maps $f\colon X\ra Y$ where $X$ and
$Y$ are homotopy equivalent to motivic spheres but not actual spheres
themselves.  In this case one can obtain a corresponding element $[f]$
of $\pi_{*,*}(S)$, but only after making specific choices of
orientations for $X$ and $Y$.

To make this precise, let us say that a \dfn{homotopy sphere} is a
motivic spectrum $X$ that is isomorphic to some sphere $S^{p,q}$ in
the motivic stable homotopy category.  An 
\dfn{oriented homotopy sphere} is a motivic spectrum
$X$ together with a {\it specified\/} isomorphism $X \map S^{p,q}$ in
the motivic stable homotopy category.

A given homotopy sphere has many orientations.  The set of
orientations is in bijective correspondence with the set of
multiplicative units inside $\pi_{0,0}(S)$.  We call this set of units
the \dfn{motivic orientation group}, which depends on the
base scheme in general.  Note
that the analog in classical topology is the group $\Z/2=\{-1,1\}$.
By Morel's Theorem, 
over perfect fields whose characteristic is not $2$,
the motivic orientation group is the group of units
in the Grothendieck-Witt ring $GW(k)$.  A formulaic description of
this group seems to be unknown, but we do not actually need to know
anything specific about it for the content of this paper.
Nevertheless, understanding this group is a curious problem  and so we
do offer the remark below:

\begin{remark}[Motivic orientations]
Recall that $\GW(k)$ is obtained by quotienting the free algebra on
symbols $\lab{a}$ for $a\in k^\times$ by the relations
\begin{enumerate}
\item
$\lab{a}\lab{b}=\lab{ab}$.
\item 
$\lab{a^2}=1$.
\item
$\lab{a}+\lab{b}=\lab{a+b}+\lab{ab(a+b)}$.
\end{enumerate}
The elements $\lab{a}$ are
clearly units in $\GW(k)$, and so one obtains a group map $\Z/2 \times
\bigl [ k^\times/(k^\times)^2\bigr ]
\ra \GW(k)^\times$ by sending the generator of $\Z/2$ to $-1$
and the element $[a]$ of $k^\times/(k^\times)^2$ to $\lab{a}$.

The map $\Z/2\times \bigl [k^\times/(k^\times)^2\bigr ]\ra \GW(k)^\times$ is an isomorphism in some
cases like $k=\R$ and $k=\Q_p$ for $p\equiv 1$ mod $4$, but not in
other cases like $k=\Q_p$ for
$p\equiv 3$ mod $4$. 
\end{remark}

\begin{remark}[Suspension data]
\label{rem:susp-data}
Here is a fundamental difficulty that occurs even in classical
homotopy theory:
given two objects $A$ and $B$ that are models for the suspension of an
object $X$, there is no canonical isomorphism between $A$ and $B$ in
the homotopy category.  In some sense, the problem boils down to the
fact that if we are just handed a model of $\Sigma X$ then we are
likely to see two cones on $X$ glued together, but we do not know
which is the ``top'' cone and which is the ``bottom''.  Mixing the
roles of the two cones tends to alter maps by a factor of $-1$.  So
when talking about models for $\Sigma X$ it is important to have the
two cones distinguished.  We define \dfn{suspension data} for $X$ to
be a diagram $[C_+X \lla X \lra C_-X]$ where both maps are
cofibrations and both $C_+X$ and $C_-X$ are contractible.  We call
$C_+X$ the {\it top cone\/} and $C_-X$ the {\it bottom cone\/}.
Choices of suspension data will appear throughout the paper, starting
in Remark \ref{re:orient-const}(2) below.  See
Appendix~\ref{se:canon} for more discussion of this and related issues.  
\end{remark}

\begin{remark}[Induced orientations on constructions]
\label{re:orient-const}
If $X$ and $Y$ are homotopy spheres then constructions like suspension,
smash product, and the join (see Appendix~\ref{se:joins}) yield
 other homotopy
spheres.  If $X$ and $Y$ are oriented then these constructions
inherit orientations in a specified way:
\begin{enumerate}[(1)]
\item
If $X$ and $Y$ have orientations $X \map S^{p,q}$ and
$Y \map S^{a,b}$, then $X \Smash Y$ has an induced orientation
\[
X \Smash Y \lra S^{p,q} \Smash S^{a,b} \llra{\phi} S^{p+a,q+b},
\]
where the second map is the canoncal isomorphism.
\item
If $X$ has an orientation $X \map S^{p,q}$ and
$C_+X \lla X \lra C_-X$ constitutes suspension data for $X$,
then $C_+X \amalg_X C_-X$ has an induced orientation
\[
C_+X \amalg_{X} C_-X \lra S^{1,0} \Smash X \lra S^{1,0}\Smash S^{p,q}
\llra{\phi} S^{p+1,q},
\]
where the first map is the canonical isomorphism 
in the homotopy category (see Section~\ref{se:canon}).
\item
Suppose that $X \map S^{p,q}$ and $Y \map S^{a,b}$ are orientations.  
Then the join $X*Y$ (see Appendix~\ref{se:joins})
has an induced orientation
\[
X*Y \lra S^{1,0} \Smash X \Smash Y \llra{\he} S^{1,0}\Smash S^{p,q}\Smash
S^{a,b}
\llra{\phi} S^{p+a+1,q+b},
\]
where the first map is the canonical isomorphism in the homotopy
category from Lemma \ref{lem:join-we}.
\end{enumerate}
\end{remark}

A map $f\colon X\ra Y$ between homotopy spheres does not by itself yield an
element of $\pi_{*,*}(S)$.  But once $X$ and $Y$ are oriented we
obtain the composite
\[ 
\xymatrix{
S^{a,b} \ar[r]^\iso & X  \ar[r] & Y \ar[r]^-\cong & S^{p,q}. }
 \]
If $\tilde{f}$ denotes this composite, then $[\tilde{f}]$ gives a corresponding
element of $\pi_{a-p,b-q}(S)$.  In the future we will just denote this
element by $[f]$, by abuse of notation.

There is an important case in which one does not have to worry about
orientations:

\begin{lemma}
Let $X$ be a homotopy sphere and $f\colon X \map X$ be a self-map.
Then $f$ represents a well-defined element of $\pi_{0,0}(S)$ that is
independent of the choice of orientation on $X$.
\end{lemma}

\begin{proof}
Choose any two orientations
$g\colon X \lra S^{p,q}$ and $h\colon X \lra S^{p,q}$.
The diagram
\[
\xymatrix{
S^{p,q} \ar[r]^{g^{-1}} \ar[d]_{hg^{-1}} & X \ar[r]^f & X \ar[r]^g &
  S^{p,q} \ar[d]^{hg^{-1}} \\
S^{p,q} \ar[r]^{h^{-1}} & X \ar[r]^f & X \ar[r]^h & S^{p,q}
}
\]
commutes in the stable homotopy category.
This shows that the elements of $\pi_{0,0}(S)$ represented 
by the top and bottom rows are related by conjugation by the
element $hg^{-1}$ of $\pi_{0,0}(S)$.  But 
the ring $\pi_{0,0}(S)$ is commutative, so conjugation
by $hg^{-1}$ acts as the identity.
\end{proof}

\begin{ex}
\label{ex:orient}
The following examples specify standard orientations for the
models of spheres that we commonly encounter.

\begin{enumerate}[(1)]
\item
$\P^1$ based at $[1:1]$ or $[0:1]$.

By the {\it standard affine covering diagram\/} of $\P^1$ we mean $U_1 \la
U_1\cap U_2 \ra U_2$ where $U_1$ (resp., $U_2$) is the open subscheme of points
$[x:y]$ such that $x\neq 0$ (resp., $y\neq 0$).  There is an evident
isomorphism of diagrams
\[ \xymatrix{U_1\ar[d]_\iso & U_1\cap U_2\ar[r]\ar[l]\ar[d]_\iso & U_2\ar[d]^\iso \\
\A^1 & (\A^1-0)\ar[l]_-{i}\ar[r]^-{\inv} & \A^1
}
\]
where $i$ is the inclusion and $\inv$  sends a point $x$ to $x^{-1}$.  (For
example, $U_1\ra \A^1$ sends $[x:y]$ to $\frac{y}{x}$).  
The bottom row of the diagram is suspension data
for $(\A^1-0)$.  As in Remark~\ref{re:orient-const},
this gives an orientation to $\A^1 \amalg_{(\A^1-0)} \A^1$.
The canonical map $\A^1 \amalg_{(\A^1-0)} \A^1 \ra \P^1$ is a weak
equivalence, which gives an orientation on $\P^1$ as well.

\item $(\A^n-0)$ based at $(1,1,\ldots,1)$ or at $(1,0,\ldots,0)$.

We orient $(\A^n-0)$ as the join of $(\A^1-0)$ and $(\A^{n-1}-0)$.  That is
to say, 
$\bigl [ (\A^1-0) \times \A^{n-1}\bigr ] \amalg_{(\A^1-0) \times (\A^{n-1}-0)} 
   \bigl [\A^1 \times (\A^{n-1}-0)\bigr ]$ has an orientation using
Remark~\ref{re:orient-const} and induction.
The canonical map
\[
\bigl [[(\A^1-0) \times \A^{n-1}\bigr ] \amalg_{(\A^1-0) \times (\A^{n-1}-0)} 
   \bigl [\A^1 \times (\A^{n-1}-0)\bigr ] \map (\A^n-0)
\]
is a weak equivalence,
which gives an orientation on $(\A^n-0)$.  

\item The split unit sphere $S_{2n-1}$ based at $(1,0,\ldots,0)$.

Let $\A^{2n}$ have coordinates $x_1,y_1,\ldots,x_n,y_n$ and let
$S_{2n-1}\inc \A^{2n}$ be the closed subvariety defined by
$x_1y_1+\cdots+x_ny_n=1$.  The quadratic form on the left of this
equation is called the {\it split quadratic form\/}, and $S_{2n-1}$ is called
the unit sphere with respect to this split form.  Let $\pi\colon S \ra
(\A^n-0)$ be the map $(x_1,y_1,\ldots,x_n,y_n)\mapsto
(x_1,x_2,\ldots,x_n)$.    This is a Zariski-trivial bundle with fibers
$\A^{n-1}$, and so $\pi$ is a weak equivalence (this follows
from a standard argument, for example using the techniques
of \cite[3.6--3.9]{DI1}).  The standard
orientation on $(\A^n-0)$ therefore induces an orientation on $S$ via
$\pi$.  

Note that there are other weak equivalences $S\ra (\A^n-0$), such as
$(x_1,y_1,\ldots,x_n,y_n)\mapsto (x_1,x_2,\ldots,x_{n-1},y_n)$.  These
maps can induce different orientations on $S$.
\end{enumerate}
\end{ex}


\section{Diagonal maps and power maps}

Le $X$ be an unstable, oriented, homotopy sphere 
that is equivalent to $S^{p,q}$ for some $p \geq q \geq 0$.
The diagonal $\Delta_X \colon X \ra X \Smash X$
represents an element in $\pi_{-p,-q}(S)$.  When $X=S^{p,q}$ we write
$\Delta_{p,q}=\Delta_{S^{p,q}}$.  Our goal in this section
is the following result.

\begin{thm} 
\label{th:diagonal}
Let $p \geq q \geq 0$.  The element $[\Delta_{p,q}]$ in $\pi_{-p,-q}(S)$ is zero 
if $p>q$, and it is $\rho^q$ if $p = q$.
\end{thm}

In classical algebraic topology these diagonal maps are all null
(except for $\Delta_{S^0}$)
because of the following lemma.

\begin{lemma}
\label{le:diagonal1}
If $X$ is a simplicial suspension then $\Delta_X$ is null.
\end{lemma}

\begin{proof}
We assume that $X=S^{1,0}\Smash Z$ and we
consider the commutative diagram
\[ 
\xymatrixcolsep{3.5pc}
\xymatrix{
S^{1,0}\Smash Z \ar[d]_{\Delta_X}\ar[dr]^{\Delta_{1,0} \Smash \Delta_Z} \\
S^{1,0} \Smash Z \Smash S^{1,0} \Smash Z \ar[r]^{1\Smash T\Smash 1} & S^{1,0}\Smash S^{1,0} \Smash Z \Smash Z.
}
\]
Since the horizontal map is an isomorphism in the homotopy
category, it is sufficient to check that $\Delta_{1,0}$ is null.
But this is true in the homotopy category of $\sSet_*$, and therefore is
true in pointed motivic spaces---the latter follows using the left Quillen functor $c\colon \sSet_* \ra
\sPre_*(\Sm/k)$.
\end{proof}

Our next goal will be to show that $[\Delta_{1,1}]=\rho$ in
$\pi_{-1,-1}(S)$.  We will exhibit an explicit
geometric homotopy, but to do this we will need to suspend 
so that we are looking at maps $S^{2,1}\ra S^{3,2}$ instead of
$S^{1,1}\ra S^{2,2}$.  
The point is that $(\A^2-0)$ gives a convenient geometric model for  $S^{3,2}$ 
(see Example~\ref{ex:orient}(2)).

We start by
considering the two maps
\[ \Delta_{1,1},\id \Smash \rho \colon (\A^1-0) \lra (\A^1-0)\Smash (\A^1-0).\]
We will model the suspensions of these two maps
via conveniently chosen suspension data.

Let $C$ be the pushout
$[(\A^1-0)\times \A^1] \amalg_{(\A^1-0)\times \{1\}} [\A^1\times \{1\}]$;
by left properness, $C$ is contractible (recall that all constructions
occur in the presheaf category).
Below we will provide several maps $C\ra \A^2-0$,
so note that to specify such a map it suffices  to give a polynomial
formula $(x,t)\mapsto f(x,t)=(f_1(x,t),f_2(x,t))$ with the ``formal'' properties
that $f(x,t)\neq (0,0)$ whenever $x\neq 0$, and  $f(x,1)\neq (0,0)$
for all $x$.  Rigorously, this amounts to the ideal-theoretic
conditions that  $f_1,f_2\in k[x,t]$, $x\in\Rad(f_1,f_2)$ and
$(f_1(x,1),f_2(x,1))=k[x]$.  

Let  $D= C\amalg_{(\A^1-0)} C$
where $(\A^1-0)$ is embedded in both copies of $C$ as the presheaf $(\A^1-0)\times
\{0\}$.
Note that $[C\bcof (\A^1-0) \cof C]$ is suspension data for
$(\A^1-0)$, and so $D$ is a model for the suspension of $(\A^1-0)$.  
Let us adopt the notation where we use $(x,t)$ for the coordinates
in the first copy of $C$, and $(y,s)$ for the $\A^1$-coordinate in the
second copy of $C$.  Heuristically, $D$ consists of two kinds of points
$(x,t)$ and $(y,s)$, which are identified when $s=t=0$ and $x=y$.  
Let us also use $(z,w)$ for the coordinates on the target $(\A^2-0)$. 

Define the map $\delta\colon D \ra (\A^2-0)$ by the following
formulas:
\begin{align*}
(x,t) &\mapsto (x,(1-t)x+t)=(1-t)(x,x)+t(x,1), \\
(y,s) &\mapsto ((1-s)y+s,y)=(1-s)(y,y)+s(1,y).
\end{align*}
We leave the reader to verify that these formulas do specify two
maps $C\ra (\A^2-0)$ that agree on $(\A^1-0)\times \{0\}$, and hence
determine a map $D\ra (\A^2-0)$ as claimed.  In a moment we will show
that $\delta$ gives a model for $\Sigma\Delta_{1,1}$ in the 
motivic stable homotopy
category.

Let us also define a map $R\colon D\ra (\A^2-0)$ by the formulas:
\begin{align*}
(x,t) &\mapsto (x,-1+2t), \\
(y,s) &\mapsto (y,-1).
\end{align*}
Once again, these formulas do specify two maps
$C \ra (\A^2-0)$ that agree on $(\A^1-0)\times \{0\}$, and
hence determine a map $D \ra (\A^2-0)$.
We will show that $R$ gives a model for
$\Sigma (\id \Smash \rho)$.

Because the formulas are rather unenlightening, we give pictures
that depict the maps $\delta$ and $R$.
Each picture shows a map $D\ra (\A^2-0)$,
with the two different shadings representing the image of the ``top''
and ``bottom'' halves of $D$, i.e., the two copies of $C$.  
Note that the pictures have been drawn
as if the second coordinate on $C$ were an interval $[0,1]$
instead of $\A^1$.  Really, the two shaded regions
should each stretch out infinitely in both directions; but this would produce
a picture with too much overlap to be useful.

\begin{picture}(300,120)
\put(45,20){\includegraphics[scale=0.3]{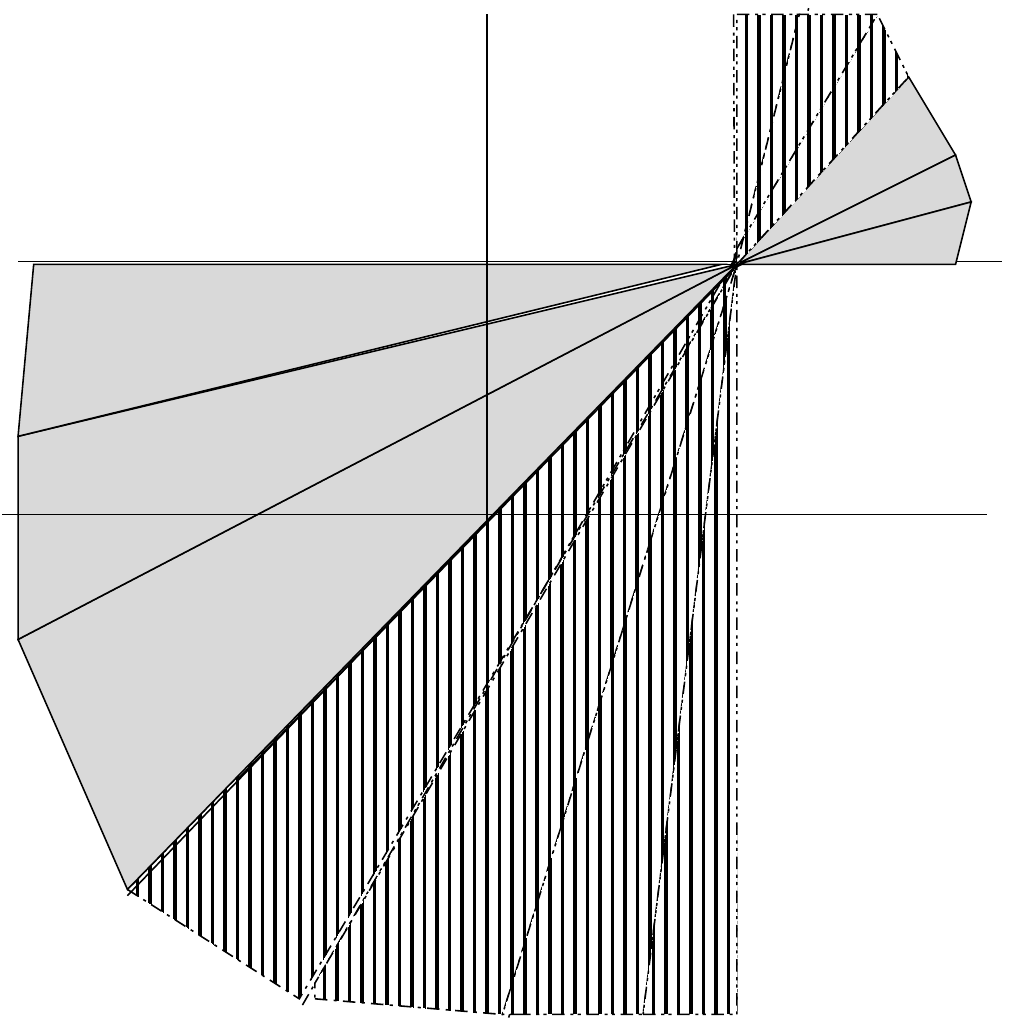}}
\put(200,20){\includegraphics[scale=0.3]{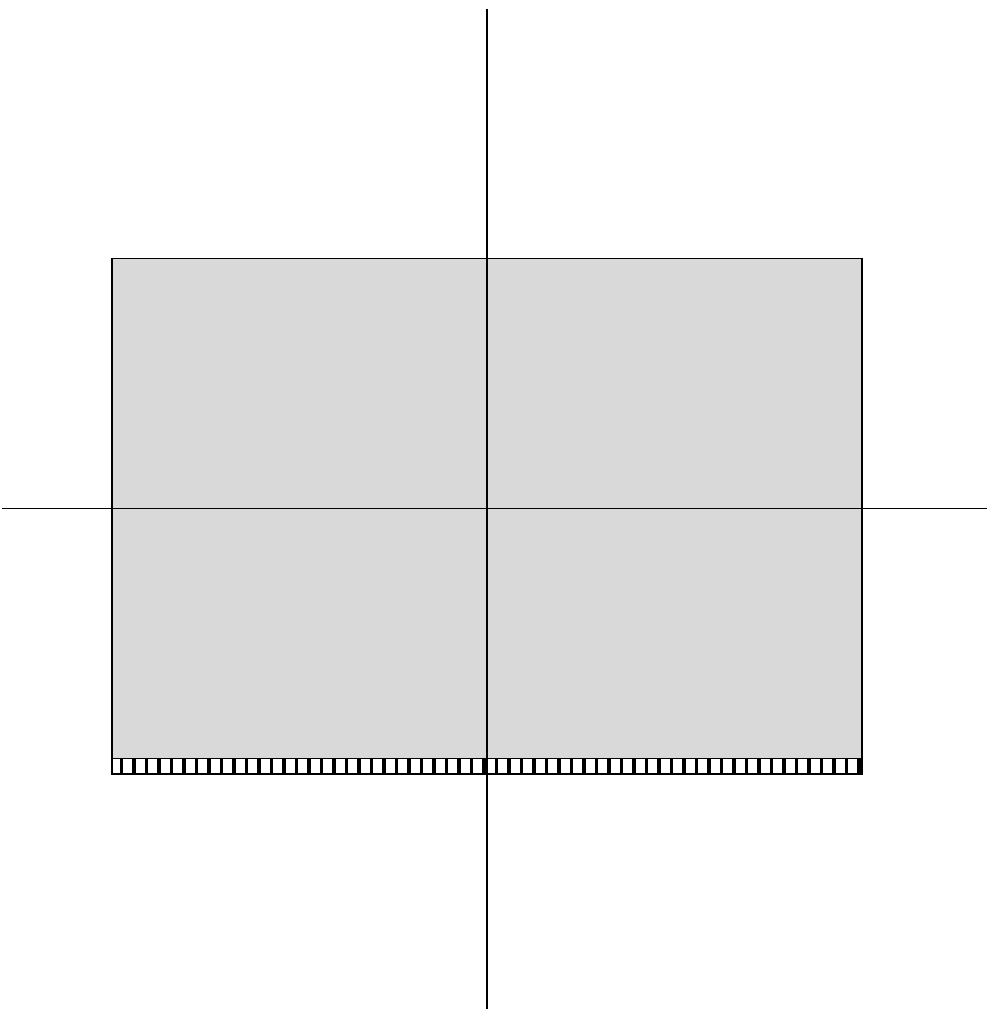}}
\put(194,35){$R$}
\put(34,35){$\delta$}
\end{picture}

In the picture for $\delta$ the reader should note the diagonal
$(\A^1-0) \inc (\A^2-0)$, and the picture should be interpreted as giving
two deformations of the diagonal: one deformation swings the punctured-line
about $(1,1)$ until it becomes the $w=1$ punctured-line (at which time
it can be ``filled in'' to an $\A^1$, not just an $\A^1-0$).  The
second swings the punctured line in the other direction until it
becomes $z=1$, and again is filled in to an $\A^1$ at that time.
This is our map $\delta\colon D\ra (\A^2-0)$.  

The picture for $R$ is simpler to interpret.  We map $\A^1-0$ to
$\A^2-0$ via $x\mapsto (x,-1)$; one deformation moves this vertically
up to $x\mapsto (x,1)$ and then fills it in to a map from $\A^1$,
whereas the second deformation leaves it constant and then fills it in.
This gives us two maps $C\ra \A^2-0$ which patch together to define
$R\colon D\ra \A^2-0$.  

Having introduced $\delta$ and $R$, our next step is to show that
they represent the elements $[\Delta_{1,1}]$ and $\rho$ in $\pi_{-1,-1}(S)$.

\begin{lemma}
\label{lem:delta}
\mbox{}
\begin{enumerate}[(a)]
\item
The map $\delta\colon D \map (\A^2-0)$
represents $[\Delta_{1,1}]$ in $\pi_{-1,-1}(S)$.
\item
The map $R\colon D \map (\A^2-0)$
represents $\rho$ in $\pi_{-1,-1}(S)$.
\end{enumerate}
\end{lemma}

\begin{proof}
For $\delta$, consider the diagram
\[
\xymatrixcolsep{1.2pc}
\xymatrix{
& \bigl [ (\A^1-0)\Smash\A^1 \bigr ]\amalg_{(\A^1-0)\Smash (\A^1-0)} 
\bigl [\A^1\Smash (\A^1-0)\bigr ]
   \ar[dr] \\
C \amalg_{(\A^1-0)} C \ar[ur]^{\delta'}\ar[dr]_\delta &
\bigl [ (\A^1-0)\times\A^1\bigr ] \amalg_{(\A^1-0)\times (\A^1-0)}
\bigl[ \A^1\times (\A^1-0)\bigr ]
   \ar[u]\ar[r]\ar[d] & \frac{\A^2-0}{\A^1 \times 1 \cup 1 \times \A^1} \\
& \A^2-0. \ar[ur] }
\]
Here $\delta'$ takes the first copy of $C$ to $(\A^1-0) \Smash \A^1$
via $(x,t) \mapsto (x,(1-t)x+t)$,
and takes the second copy of $C$ to $\A^1 \Smash (\A^1-0) $
via the formula $(y,s) \mapsto ((1-s)y+s,y)$.
The smash products are important; they allow us to define $\delta'$
on the two copies of $\A^1 \times \{1\}$ in the two copies of $C$.  
Note that the outer parallelogram obviously commutes.  

The five unlabeled arrows are all weak equivalences between homotopy
spheres.  Orient each of the spheres in the following way:
\begin{itemize}
\item The top sphere is oriented as the suspension of $(\A^1-0) \Smash
(\A^1-0)$.
\item The middle sphere is oriented as the join $(\A^1-0) * (\A^1-0)$.
\item $\A^2-0$ is oriented in the standard way.
\item $(\A^2-0)/(\A^1\times 1\cup 1\times \A^1)$ is oriented so that
the projection map from $\A^2-0$ is orientation-preserving.
\end{itemize}
The five weak equivalences are then readily checked to be
orientation-preserving.  This part of the argument only serves to
verify that the standard orientations on the top and bottom spaces (in
the middle column) match up when we map to $[\A^2-0]/(\A^1\times 1\cup
1\times \A^1)$.  

The commutativity of the outer parallelogram now implies that $\delta$ and
$\delta'$ represent the same element of $\pi_{-1,-1}(S)$.
Since $\delta'$ is clearly a model for the suspension of $\Delta_{1,1}$,
this completes the proof that $[\delta]=[\Delta_{1,1}]$.  

The same argument shows that $R$ is a model for the suspension of 
$\id \Smash R$.  Here, we use a map $R'$ that takes the first copy
of $C$ to $\A^1 \Smash (\A^1-0)$ via $(x,t) \mapsto (x,-1+2t)$, and takes
the second copy of $C$ to $(\A^1-0) \Smash \A^1$ via
$(y,s) \mapsto (y,-1)$.
\end{proof}

Our final step is to show that $\delta$ and $R$ are homotopic.
We will give a series of homotopies that deforms $\delta$ to $R$.
Since the
formulas are again rather unenlightening we begin by giving a sequence of
pictures that depict three intermediate stages.
Each is a map $D\ra (\A^2-0)$; we will then
give four homotopies showing how to deform each picture to the next.
These pictures will not make complete sense until one compares to the
formulas in the arguments below, but it is nevertheless useful to see
the pictures ahead of time. 

\begin{picture}(340,240)(0,20)
\put(0,150){\includegraphics[scale=0.3]{homot0.pdf}}
\put(120,150){\includegraphics[scale=0.3]{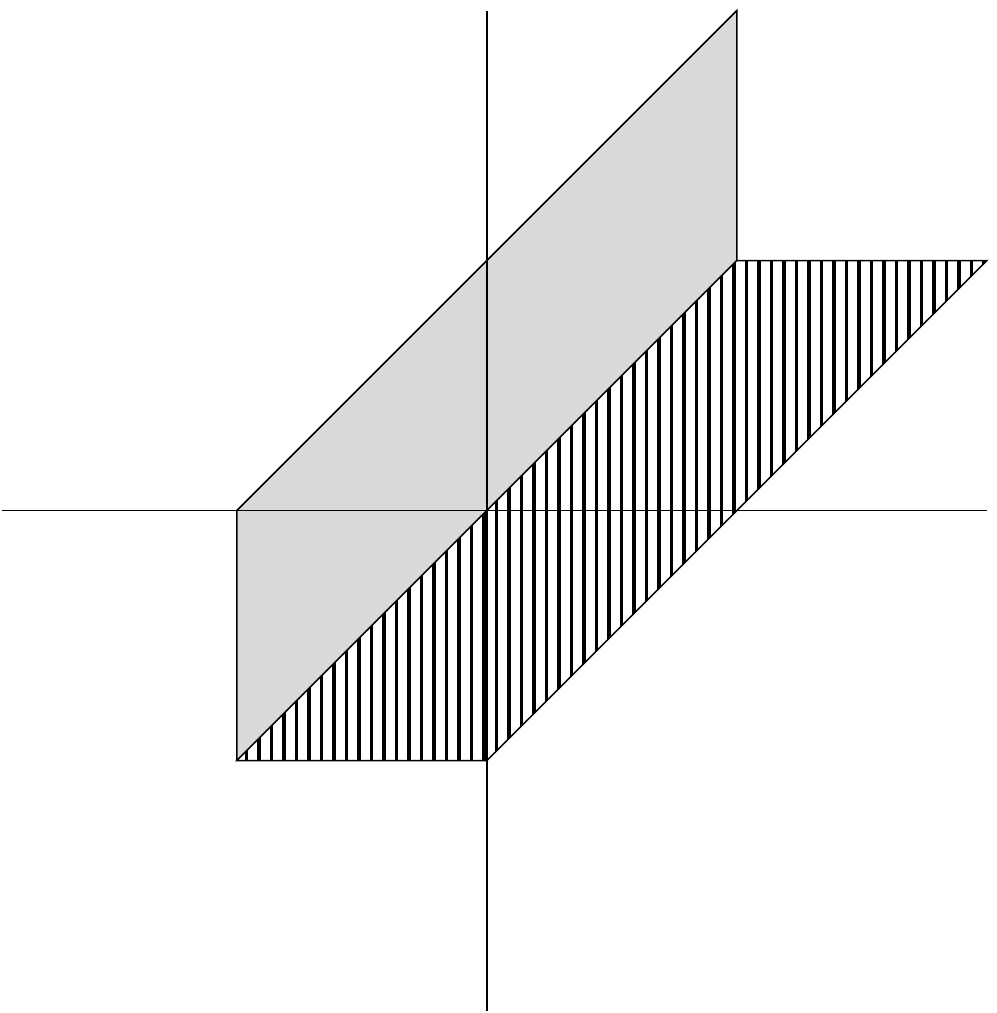}}
\put(240,150){\includegraphics[scale=0.3]{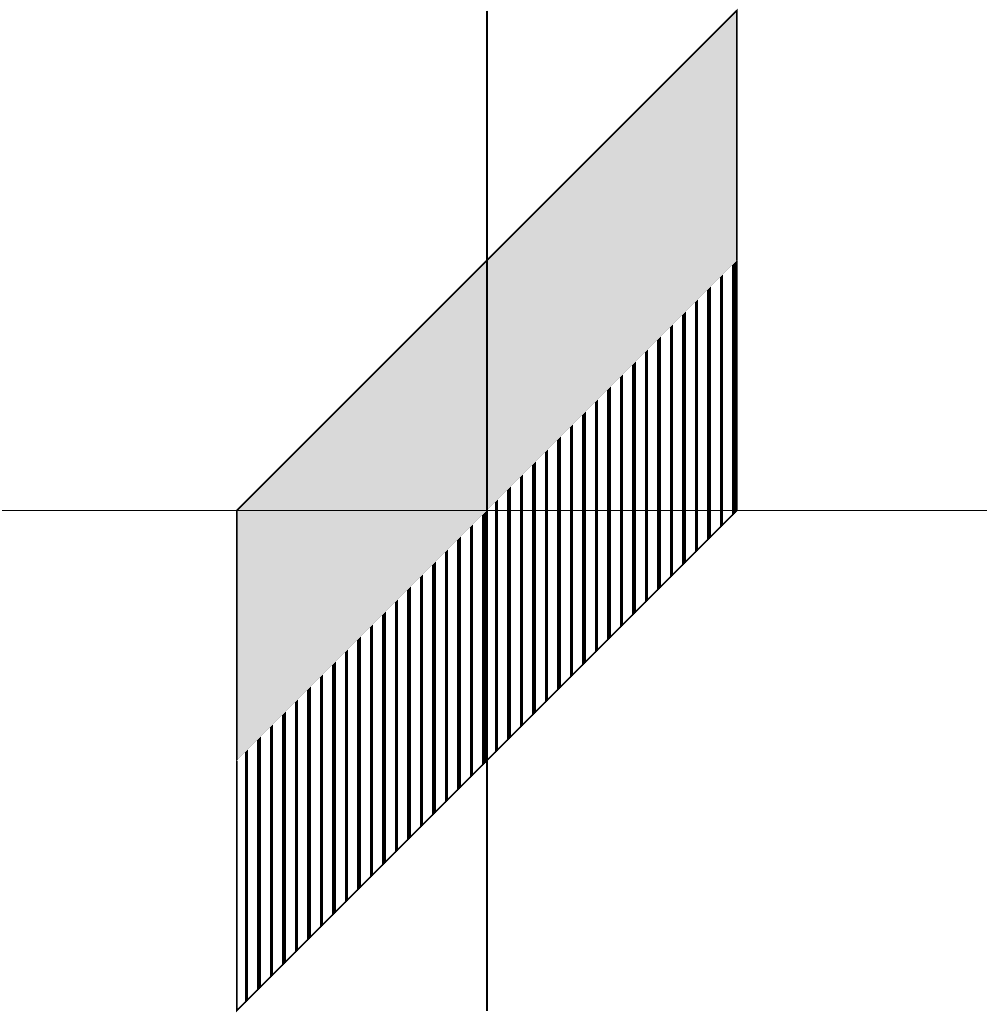}}
\put(190,30){\includegraphics[scale=0.3]{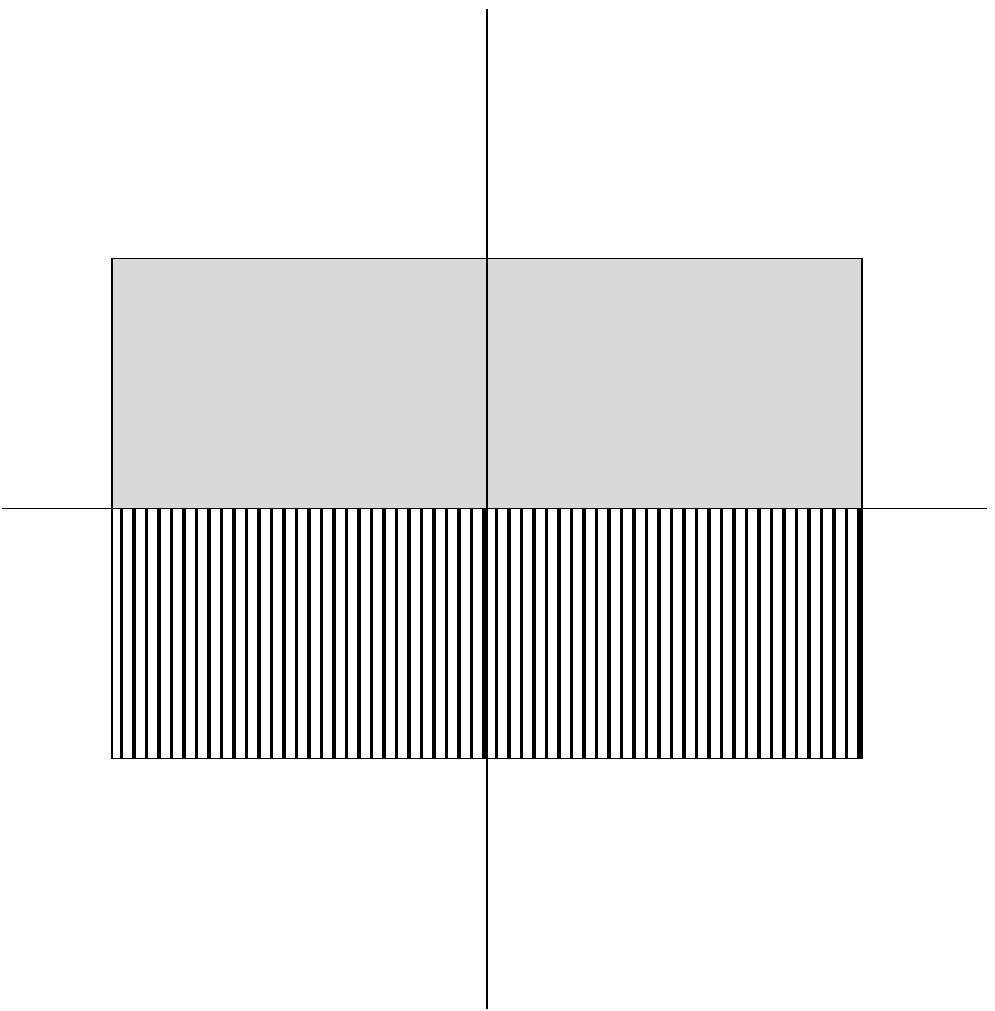}}
\put(30,30){\includegraphics[scale=0.3]{homot4.pdf}}
\put(190,170){$f_1$}
\put(310,170){$f_2$}
\put(23,35){$R$}
\put(193,35){$f_3$}
\put(70,170){$\delta$}
\put(142,72){$\lparrow$}
\put(92,193){$\parrow$}
\put(212,193){$\parrow$}
\put(270,135){$\swparrow$}
\end{picture}

These pictures (and the explicit formulas in the proof below) can
seem unmotivated.  It is useful to know that
each homotopy is a standard straight-line homotopy.  If one has the
idea of deforming $\delta$ to $R$, and the only tool one is allowed to
use is a
straight-line homotopy, a bit of stumbling around quickly leads to
the above chain of maps; there is nothing deep here.  

\begin{lemma}
\label{le:delta-R}
The maps 
$\delta$ and $R$ are homotopic as unbased maps
$D \map (\A^2-0)$.
\end{lemma}

\begin{proof}
We will give a sequence of maps $H\colon D\times \A^1\ra (\A^2-0)$,
each giving an $\A^1$-homotopy from $H_0$ to $H_1$.  These will
assemble into a chain of homotopies from $\delta$ to $R$.   Note that
$D\times \A^1$ is isomorphic to 
$(C\times \A^1)\amalg_{(\A^1-0)\times \A^1} (C\times \A^1)$ 
and that $C\times \A^1$ is isomorphic to
\[ 
\bigl [(\A^1-0)\times \A^1\times \A^1\bigr ]
\amalg_{(\A^1-0)\times \{1\}\times \A^1}
\bigl [ \A^1\times \{1\}\times \A^1 \bigr ].
\]
To specify a map $C\times \A^1\ra (\A^2-0)$,
it suffices to
give a polynomial formula
$(x,t,u) \mapsto f(x,t,u) = (f_1(x,t,u), f_2(x,t,u))$
with the ``formal'' properties that
$f(x,t,u) \neq (0,0)$ whenever $x \neq 0$, and
$f(x,1,u) \neq (0,0)$ for all $x$ and $u$.
Rigorously, this amounts to the ideal-theoretic conditions that
$f_1, f_2 \in k[x,t,u]$,
$x\in \Rad(f_1,f_2)$ and $(f_1(x,1,u),f_2(x,1,u))=k[x,u]$.

Here are three maps $D\ra (\A^2-0)$:
\begin{align*}
& f_1\colon \qquad (x,t) \mapsto (x,x+t), &
(y,s) \mapsto (y+s,y) \\
& f_2\colon \qquad (x,t) \mapsto (x,x+t), &
(y,s) \mapsto (y,y-s) \\
& f_3\colon \qquad (x,t) \mapsto (x,t), & 
(y,s) \mapsto (y,-s)
\end{align*}
and
here are four homotopies:
\begin{align*}
H_1&: \begin{cases}
\quad(x,t,u) &\mapsto (x, (1-t+ut)x+t) \\
\quad(y,s,u) &\mapsto ( (1-s+us)y + s, y).
\end{cases}
\\
H_2&: 
\begin{cases}
\quad (x,t,u) &\mapsto (x, x+t),\\
\quad (y,s,u) &\mapsto ( y + (1-u)s, y - us). 
\end{cases}
\\
H_3&:
\begin{cases}
\quad (x,t,u) &\mapsto (x, x+t-ux) \\
\quad (y,s,u) &\mapsto (y, y-s-uy). 
\end{cases}
\\
H_4&:
\begin{cases}
\quad (x,t,u) &\mapsto (x, (1-u)t+u(2t-1)) \\
\quad (y,s,u) &\mapsto (y, (u-1)s-u). 
\end{cases}
\end{align*}
We leave it to the reader to verify that each formula really does
define a map $D\times \A^1\ra (\A^2-0)$, and that these give
$\A^1$-homotopies
\[ \delta \he f_1 \he f_2 \he f_3 \he R.
\]
\end{proof}

\begin{prop}
\label{pr:diag}
The diagonal map $(\A^1-0) \map (\A^1-0)\Smash(\A^1-0)$
represents $\rho$ in $\pi_{-1,-1}(S)$.
\end{prop}

\begin{proof}
Consider the two maps $\delta,R\colon D\ra \A^2-0$. These are based
maps if $D$ and $\A^2-0$ are both given the basepoint $(1,1)$ (in the
case of $D$, choose the point $(1,1)$ in the first copy of $C$).  
Since $\delta$ and $R$ are unbased homotopic,
Lemma~\ref{le:unbased-sigma} below yields that $\Sigma^\infty
\delta=\Sigma^\infty R$ in the stable homotopy category.
We therefore obtain $[\Delta]=[\delta]=[R]=\rho$,
with the first and last equalities by  Lemma~\ref{lem:delta}.
\end{proof}

\begin{lemma}
\label{le:unbased-sigma}
Let $X$ and $Y$ be pointed motivic spaces, and let $f,g\colon X\ra Y$ be two
maps.  If $f$ and $g$ are homotopic as unbased maps then
$\Sigma^\infty f=\Sigma^\infty g$ in the motivic stable homotopy category.
\end{lemma}

\begin{proof}
For any motivic space $A$, 
let $C_u A=[A\times c(\Delta^1)]/[A\times \{1\}]$ denote 
the unbased simplicial cone on a space $A$, and let $\Sigma_u A$ be
the unbased suspension functor 
\[ \Sigma_u A=(C_uA)\amalg_{A\times \{0\}} (C_uA).
\]
Equip $\Sigma_u A$ with the basepoint given by the ``cone point'' in
the first copy of $C_uA$.  
When $A$ is pointed, let $\Sigma A$ be the usual based simplicial
suspension, i.e. $\Sigma A=(\Sigma_u A)/(\Sigma_u *)$.  Note that
the projection $\Sigma_uA\ra \Sigma A$ is a natural based motivic weak
equivalence.

Applying $\Sigma_u$ and $\Sigma$ to $f$ and $g$, and then stabilizing via
$\Sigma^\infty(\blank)$, yields
the diagram
\[ \xymatrixcolsep{3pc}\xymatrix{
\Sigma^\infty(\Sigma_uX) \ar@<0.5ex>[r]^{\Sigma^\infty(\Sigma_uf)}
\ar@<-0.5ex>[r]_{\Sigma^\infty(\Sigma_u g)}
\ar[d]_\he & \Sigma^\infty(\Sigma_u
Y)\ar[d]^\he \\
\Sigma^\infty(\Sigma X)\ar@<0.5ex>[r]^{\Sigma^\infty(\Sigma f)}
\ar@<-0.5ex>[r]_{\Sigma^\infty(\Sigma g)}
& \Sigma^\infty(\Sigma Y).
}
\]
Since $f$ and $g$ are unbased homotopic,
$\Sigma_u f$ and $\Sigma_u g$ are based homotopic.  
Hence $\Sigma^\infty(\Sigma_u f)=\Sigma^\infty(\Sigma_u g)$ 
in the stable homotopy category.  
The above diagram then shows that $\Sigma^\infty(\Sigma
f)=\Sigma^\infty(\Sigma g)$, and hence $\Sigma^\infty f=\Sigma^\infty g$.
\end{proof}

Our identification of $[\Delta_{1,1}]$ with $\rho$ gives a nice geometric
explanation for the following relation in $\pi_{*,*}(S)$.

\begin{cor}
\label{cor:epsilon-rho}
In $\pi_{*,*}(S)$, there is the relation $\epsilon \rho=\rho
\epsilon=\rho$.  
\end{cor}

\begin{proof}
We already know from Proposition~\ref{pr:commute}
 that $\epsilon$ is central, as it lies in
$\pi_{0,0}(S)$.   
Using the model for $\epsilon$ as the twist map on $S^{1,1}$,
note that $\epsilon
\Delta_{1,1}=\Delta_{1,1}$ as maps $S^{1,1}\ra S^{1,1}\Smash S^{1,1}$
\end{proof}

We can finally conclude the proof of the main result of this section.

\begin{proof}[Proof of Theorem~\ref{th:diagonal}]
If $p>q$ then $S^{p,q}$ is a simplicial suspension, and so
$\Delta_{p,q}$ is null by Lemma~\ref{le:diagonal1}.  For the case
where $p=q$ we consider the commutative diagram
\[ \xymatrixcolsep{2.5pc}\xymatrix{
S^{q,q} \ar@{=}[r] \ar[d]_{\Delta_{q,q}} & S^{1,1} \Smash \cdots \Smash S^{1,1}
\ar[d]^{\Delta_{1,1}\Smash \cdots \Smash \Delta_{1,1}} \\
S^{q,q}\Smash S^{q,q} \ar[r]^-T & (S^{1,1}\Smash S^{1,1})\Smash \cdots
\Smash (S^{1,1}\Smash S^{1,1})
}
\]
where $T$ is an appropriate composition of twist and associativity
maps, involving $\binom{q}{2}$ twists.  Note that
$[T]=\epsilon^{\binom{q}{2}}$, using Remark~\ref{re:invcoh}.  The
diagram then gives
\[ \rho^q=[\Delta_{1,1}]^q=[T]\cdot
[\Delta_{q,q}]=\epsilon^{\binom{q}{2}}[\Delta_{q,q}]
\]
using Proposition~\ref{pr:diag} for the first equality.  Rearranging
gives $[\Delta_{q,q}]=\epsilon^{\binom{q}{2}}\rho^q=\rho^q$, using
$\epsilon \rho=\rho$ in the final step.
\end{proof}

The following result 
about arbitrary motivic homotopy ring spectra
is a direct consequence of our work above.

\begin{cor}
\label{cor:rho-ring-spectra}
Let $E$ be a motivic homotopy ring spectrum, i.e., a monoid in the motivic
stable homotopy category.  Write $\bar{\rho}$ for
the image of $\rho$ under the unit map $\pi_{*,*}(S) \ra \pi_{*,*}(E)$.  
For each $n\geq 0$,
there is an isomorphism
$E^{*,*}(S^{n,n})\iso E^{*,*}\oplus E^{*,*} x$ as $E^{*,*}$-modules,
where $x$ is a generator of bidegree $(n,n)$.  The ring structure is completely
determined by graded commutativity in the sense of
Proposition~\ref{pr:commute} together with the fact that
$x^2=\bar{\rho}^nx$.
\end{cor}

In other words, the ring $E^{*,*}(S^{n,n})$ is an $\epsilon$-graded-commutative
$E^{*,*}$-algebra on one generator $x$ of bidegree $(n,n)$, subject to the 
single relation $x^2 = \bar{\rho}^n x$.

\begin{proof}
The statement about $E^{*,*}(S^{n,n})$ as an $E^{*,*}$-module is
formal; the generator $x$ is the map
$S^{n,n}\iso S^{n,n}\Smash S^{0,0}\llra{\id\Smash u} S^{n,n}\Smash E$,
where $u\colon S^{0,0}\ra E$ is the unit map.  The graded
commutativity of $E^{*,*}(S^{n,n})$ is by \cite[Remark 6.14]{D}.
It only remains to calculate $x^2$, which is the composite
\[\xymatrixcolsep{2pc}\xymatrix{
 S^{n,n}\ar[r]^-{\Delta} & S^{n,n}\Smash S^{n,n} \ar[r]^-{x\Smash x}
& (S^{n,n}\Smash E)\Smash (S^{n,n}\Smash E) \ar[r]^-{1\Smash T\Smash 1}
& S^{n,n}\Smash S^{n,n}\Smash E\Smash E \ar[d]^{\phi\Smash \mu}\\
&&& S^{2n,2n}\Smash E.
}
\]
It is useful to write this as $h(x\Smash x)\Delta$ where
$h=(\phi\Smash \mu)(1\Smash T\Smash 1)$.

Let $f\colon S^{0,0}\ra S^{n,n}$ be a map  representing
$\rho^n$.
Then the class $\bar{\rho}^n$ in $E^{n,n}$ is represented by the
composite
\[ S^{0,0}\llra{f} S^{n,n} \iso S^{n,n}\Smash S^{0,0}\llra{\id\Smash u}
S^{n,n}\Smash E,
\]
which can also be written as $x\circ f$.  
So $\bar{\rho}^n \cdot x$ is represented by the composite
\[ S^{n,n}\iso S^{0,0}\Smash S^{n,n} \llra{xf\Smash x} (S^{n,n}\Smash E) \Smash
(S^{n,n}\Smash E) \llra{h} S^{2n,n}\Smash E
\]
Note that the first two maps in this composite may be written as
\[ S^{n,n}\iso S^{0,0}\Smash S^{n,n} \llra{f\Smash 1} S^{n,n}\Smash
S^{n,n} \llra{x\Smash x} (S^{n,n}\Smash E)\Smash (S^{n,n}\Smash E).
\]

Comparing our representation of $x^2$ to that of $\bar{\rho}^n\cdot
x$, to show that they are equal it will suffice to prove  that $\Delta\colon S^{n,n}\ra
S^{n,n}\Smash S^{n,n}$ represents the same homotopy class as
$S^{n,n}\iso S^{0,0}\Smash S^{n,n}\llra{f\Smash \id_{n,n}} S^{n,n}\Smash
S^{n,n}$.  That is, we must verify that $[\Delta]=[f\Smash \id_{n,n}]$.  

But $[f\Smash \id_{n,n}]$
 equals $\epsilon^n [f] = \epsilon^n \rho^n = \rho^n$
by Remark \ref{re:invcoh} and Corollary \ref{cor:epsilon-rho}.
This equals $[\Delta_{n,n}]$ by Theorem \ref{th:diagonal}.
\end{proof}

\begin{remark}
When $E$ is mod $2$ motivic cohomology,
Corollary \ref{cor:rho-ring-spectra} is essentially
\cite[Lemma 6.8]{Vpower}.  The proof in \cite{Vpower} uses
special properties about motivic cohomology (specifically, the
isomorphism between certain motivic cohomology groups and Milnor
$K$-theory).  Our argument avoids these difficult results.
In some sense, our proof is a universal argument that reflects the spirit
of Grothendieck's original ``motivic" philosophy.
\end{remark}

\subsection{Power maps}

The following result is a simple corollary of
Theorem~\ref{th:diagonal}.  In the remainder of the paper we will only
need to use the case $n=-1$, which could be proven more directly, but
it seems natural to include the entire result.

\begin{prop}
\label{pr:power1}
For any integer $n$, let $P_n\colon (\A^1-0)\ra (\A^1-0)$ be the 
$n$th power map $z\mapsto z^n$.
In $\pi_{0,0}(S)$ one has
\[ 
[P_n]=
\begin{cases}
\frac{n}{2}(1-\epsilon) & \text{if $n$ is even,}\\
1+\frac{n-1}{2}(1-\epsilon) & \text{if $n$ is odd.}
\end{cases}
\]
\end{prop}

\begin{proof}
We will prove that $[P_n]=1-\epsilon [P_{n-1}]$; multiplication by
$\epsilon$ then yields that $[P_{n-1}]=\epsilon-\epsilon[P_n]$.  
The main result follows by
induction (in the positive and negative directions) starting with the
trivial base case $n=0$.  

Consider the following diagram:
\[\xymatrixcolsep{2.5pc}\xymatrix{
\A^1-0 \ar[r]^-\Delta \ar[dr]_-{\Delta_s} 
& (\A^1-0) \times (\A^1-0) \ar[r]^-{\id \times
P_{n-1}} & (\A^1-0)\times (\A^1-0) \ar[r]^-\mu & \A^1-0 \\
& (\A^1-0)\Smash (\A^1-0) \ar[u]_\chi\ar[r]_{\id\Smash P_{n-1}} &
(\A^1-0)\Smash (\A^1-0). \ar[u]_\chi\ar[ur]_\eta  \\
}
\]
Here  $\chi$ is the stable splitting  of the map from the
Cartesian product to the smash product---see Appendix~\ref{se:stable-split}.
This diagram is commutative except for the left triangle, where we
have the relation
\begin{myequation}
\label{eq:Delta_s}
 \Delta=\chi \Delta_s + j (\pi_1+\pi_2) \Delta
\end{myequation}
by Lemma~\ref{lem:Hopf-diagonal},
where $\pi_1$ and $\pi_2$ are the two projections
$(\A^1-0) \times (\A^1-0) \map (\A^1-0) \vee (\A^1-0)$
and
$j: (\A^1-0) \vee (\A^1-0) \map (\A^1-0) \times (\A^1-0)$
is the usual inclusion of the wedge into the product.
Note that $P_n$ is the composition along the top of the diagram.

We now compute that
\begin{align*}
 [P_n] & =[\mu(\id\times P_{n-1})\Delta] \\
& = [\mu(\id\times P_{n-1}) (\chi
\Delta_s + j\pi_1\Delta + j\pi_2\Delta)]\\
&= [\mu(\id\times P_{n-1})\chi\Delta_s] +
[\mu(\id\times P_{n-1})j\pi_1\Delta] +
[\mu(\id\times P_{n-1})j\pi_2\Delta]\\
&= [\eta(\id\Smash P_{n-1})\Delta_s ] + [\id] + [P_{n-1}]\\
&= [\eta][P_{n-1}][\Delta_s]+1+[P_{n-1}].
\end{align*}
In the last equality we have used that $[\id\Smash
P_{n-1}]=[P_{n-1}]$, by Remark~\ref{re:invcoh}.  

Now use that $[\Delta_s] = \rho$, elements of $\pi_{0,0}(S)$ commute,
and that $\eta \rho = -(1+\epsilon)$
(see the last statement in 
Theorem~\ref{th:Morel}).  The above equation becomes
$[P_n]=-(1+\epsilon)[P_{n-1}]+1+[P_{n-1}]$, or $[P_n]=1-\epsilon[P_{n-1}]$.
\end{proof}




\section{Cayley-Dickson algebras and Hopf maps}
\label{sctn:Cayley-Dickson}

In this section we introduce the particular Cayley-Dickson algebras
needed for our work.    We then apply the general Hopf
construction from Section~\ref{se:Hopf} to define motivic Hopf
elements $\eta$, $\nu$, and $\sigma$ in $\pi_{*,*}(S)$.  We also review
some basic properties of $\eta$, due to Morel.  

\subsection{Fundamentals}
\label{subsctn:CD-fundamentals}
We begin by reviewing the notion of generalized Cayley-Dickson
algebras from \cite{A} and \cite{Sch}.
For momentary convenience, let $k$ be a field 
not of characteristic $2$;
we will explain below in Remark \ref{rem:CD-not-field}
how to deal with the integers and fields of characteristic $2$.
 
An \dfn{involutive algebra} is a $k$-vector space
$A$ equipped with a (possibly nonassociative) unital bilinear pairing
$A\times A\ra A$ and a linear anti-automorphism $(\blank)^*\colon A\ra
A$ whose square is the identity, and such that
\[ x+x^*=2t(x) \cdot 1_A, \qquad xx^*=x^*x=n(x) \cdot 1_A \]
for some linear function $t\colon A\ra k$ and some quadratic form
$n\colon A\ra k$.   Given such an algebra together with a $\gamma$ in
$k^\times$, one can form the \dfn{Cayley-Dickson double} of $A$ with
respect to $\gamma$.  This is the algebra $D_\gamma(A)$ whose underlying vector
space is $A\oplus A$ and where the multiplication and involution are given by the
formulas
\[ (a,b)\cdot (c,d)=(ac - \gamma d^*b,da+bc^*), \qquad (a,b)^*=(a^*,-b).  
\]
It is easy to check
that this is again an involutive algebra, with $t(a,b)=t(a)$ and
$n(a,b)=n(a)+\gamma n(b)$.  We will sometimes write $D(A)$ for
$D_\gamma(A)$, when the constant $\gamma$ is understood.

Because the Cayley-Dickson doubling process yields a new involutive
algebra, it can be repeated.   Let
$\und{\gamma}=(\gamma_0,\gamma_1,\gamma_2,\ldots)$ be a sequence of elements
in $k^\times$.  Start with $A_0=k$ with the trivial involution, and 
inductively define
$A_i=D_{\gamma_{i-1}}(A_{i-1})$.  This gives a sequence of
Cayley-Dickson algebras $A_0,A_1,A_2,\ldots$, where  $A_n$ has
dimension $2^n$ over $k$.  This sequence depends on the
choice of $\und{\gamma}$.  
Here are some well-known properties of these Cayley-Dickson algebras:
\begin{enumerate}[(1)]
\item $A_1$ is commutative, associative, and normed in the sense that
$n(xy)=n(x)n(y)$ for all $x$ and $y$ in $A_1$;
\item $A_2$ is associative and normed (but non-commutative in general);
\item $A_3$ is normed (but non-commutative and non-associative in general).
\end{enumerate}

The standard example of these algebras occurs 
with $k = \R$ and $\gamma_i = 1$ for all $i$.
This data gives $A_0=\R$, $A_1=\C$, $A_2=\HH$,
and $A_3=\OO$ (as well as more complicated algebras at later stages).  
In each of these algebras, the norm form $n(x)$ is the usual
sum-of-squares form on the underlying real vector space, under an
appropriate choice of basis.

Because motivic homotopy theory takes schemes, rather than rings, as
its basic objects,
 we will often make the trivial change in
point-of-view from
Cayley-Dickson {\it algebras\/} to ``Cayley-Dickson {\it
varieties\/}''.  
If $A$ is a Cayley-Dickson algebra over $k$ of dimension $2^n$, then
the associated variety is isomorphic to the affine space $\A^{2^n}$, equipped with
the corresponding bilinear map $\A^{2^n}\times \A^{2^n}\ra \A^{2^n}$ and
involution $\A^{2^n}\ra \A^{2^n}$.  We will write $A$ both for the
algebra and for the associated affine space.  
Let $S(A)$ denote the closed subvariety of $A$ defined by the equation
$n(x)=1$.  We call this subvariety the ``unit sphere'' inside of $A$,
although the word ``sphere'' should be loosely interpreted.  
If $A$ is normed, then we obtain a pairing
\[ S(A)\times S(A) \ra S(A).
\]
The rest of this section will exploit these pairings
in motivic homotopy theory.

\subsection{Cayley-Dickson algebras with split norms}
In motivic homotopy theory, the affine quadrics
$x_1^2+x_2^2+\cdots+x_n^2=1$ are not models for motivic spheres unless
the ground field contains a square root of $-1$.  
This limits the usefulness of the classical Cayley-Dickson algebras
(where $\gamma_i=1$ for all $i$), at least as far as producing
elements in $\pi_{*,*}(S)$.  Instead, we will focus on the sequence of
Cayley-Dickson algebras corresponding to
$\und{\gamma}=(-1,1,1,1,\cdots)$.  From now on let $A_i$ denote the
$i$th algebra in this sequence.  We will shortly see that the norm form in
$A_i$ is, under a suitable choice of basis, equal to the split
 form; therefore $S(A_i)$ is a model for a
motivic sphere (see Example~\ref{ex:orient}(3)).

We start by analyzing $A_1$.  This is $\A^2$ equipped with the
mutiplication $(a,b)(c,d)=(ac+bd,da+bc)$ and involution
$(a,b)^*=(a,-b)$.  With the the change of basis $(a,b)\mapsto
(a+b,a-b)$, we can write the
multiplication as $(a,b)(c,d)=(ac,bd)$ and the involution as
$(a,b)^*=(b,a)$.  In this new basis, $(1,1)$ is the identity element of $A_1$,
and the norm form is $n(a,b) = ab$.
We will abandon the ``old'' Cayley-Dickson basis for $A_1$ and from
now on always use this new basis (in essence, we simply forget that
$A_1$ came to us as $D(A_0)$).

Observe that the unit sphere $S(A_1)$ is the subvariety of $\A^2$
defined by $xy=1$, which is isomorphic to $(\A^1-0)$.  So $S(A_1)$ is
a model for $S^{1,1}$.  

\begin{remark}
\label{rem:CD-not-field}
If $2$ is not invertible in $k$, then
we cannot perform the same change-of-basis when analyzing $A_1$.  
This case includes the integers and fields of characteristic $2$.
Instead,
we can simply ignore $A_0$ altogether and rather {\it
define\/} $A_1$ to be the ring $k\times k$, together with
$A_2=D_1(A_1)$ and $A_3=D_1(A_2)$.  This is just a small shift in
perspective.  
\end{remark}

The next algebra $A_2$ is $D_1(A_1)$, which is $\A^4$ with the following
multiplication:
\begin{align*}
 (a_1,a_2,b_1,b_2)\cdot
(c_1,c_2,d_1,d_2)& =(ac-d^*b,da+bc^*)\\
&=(a_1c_1-d_2b_1,a_2c_2-d_1b_2,d_1a_1+b_1c_2,d_2a_2+b_2c_1).
\end{align*}
The involution is $(a_1,a_2,b_1,b_2)^*=(a_2,a_1,-b_1,-b_2)$, and the
norm form is readily checked to be
\[ n(a_1,a_2,b_1,b_2)=a_1a_2+b_1b_2.
\]
This is the split quadratic form on $\A^4$.

The algebra $A_3=D_1(A_2)$ has underlying variety $\A^8$.  We will not write
out the formulas for multiplication and involution here, although they
are easy enough to deduce.  The norm form on $A_3$ is once again
the split form.

The algebras $A_1$, $A_2$, and $A_3$ all have normed multiplications,
in the sense that $n(xy)=n(x)n(y)$ for all $x$ and $y$.  

It is useful to regard $A_1$, $A_2$, and $A_3$ as analogs of the
classical algebras $\C$, $\HH$, and $\OO$.  We call them the ``split
complex numbers'', the ``split quaternions'', and the ``split octonions'',
respectively.  One should not take
the comparisons too seriously: for example, $A_1$ has zero divisors 
 whereas $\C$ is a field.  
Nevertheless, they are normed algebras that turn out to play
roles in motivic homotopy that are entirely analogous to the roles
that $\C$, $\HH$, and $\OO$ play in ordinary homotopy theory. 
We will adopt the notation
\[ A_\C=A_1, \quad  A_\HH=A_2, \quad A_\OO=A_3; 
\]
\[
S_\C=S(A_\C), \ S_\HH=S(A_\HH),\ S_\OO=S(A_\OO).
\]
The multiplications in $A_\C$, $A_\HH$, and $A_\OO$ restrict to
give pairings $S_\C\times S_\C\ra S_\C$, $S_\HH\times S_\HH \ra
S_\HH$, and $S_\OO\times S_\OO\ra S_\OO$.  
Note that
\[ S_\C\he S^{1,1}, \quad S_\HH\he S^{3,2}, \quad S_\OO\he S^{7,4}.
\]
More generally, $S(A_n)$ is a model for $S^{2^n-1,2^{n-1}}$ for all $n$.

Recall the isomorphism $S_\C\llra{\he} \A^1-0$ given by  $(a_1,a_2)\mapsto
a_1$.  This provides an orientation on $S_\C$.
Under this isomorphism, the product $S_\C\times S_\C\ra S_\C$
coincides with the usual multiplication map on $(\A^1-0)$.

We orient $S_\HH$ via the weak equivalence 
$S_\HH\llra{\he} (\A^2-0)$ that sends $(a_1,a_2,b_1,b_2)$ to $(a_1,b_1)$,
using the standard orientation on $(\A^2-0)$ from Example \ref{ex:orient}(2).
Similarly, we orient $S_\OO$ via the analogous weak equivalence
$S_\OO\llra{\he} (\A^4-0)$.

\begin{remark}
\label{rem:SL2}
The algebra $A_\HH$ is isomorphic to the algebra  of 
$2 \times 2$ matrices, via the isomorphism
\[ (a_1,a_2,b_1,b_2)\mapsto \begin{bmatrix} 
a_1 & b_1 \\
-b_2 & a_2
\end{bmatrix}.
\]
This is easy to check using the formula for multiplication in $A_\HH$.
Under this isomorphism, the conjugate of a matrix $A$ corresponds to 
the classical adjoint of $A$, i.e.,
\[
M=\begin{bmatrix} a & b \\ c & d \end{bmatrix}
\mapsto
M^*=\adj M=\begin{bmatrix} d & -b \\ -c & a \end{bmatrix}.
\]
The norm form is equal to the  determinant, and consequently one has
$S_\HH\iso \SL_2$.  The pairing $S_\HH\times S_\HH\ra S_\HH$ is just
the usual product on $\SL_2$.  
\end{remark}

\begin{remark}
\label{re:splitting-H}
The `splitting' of the algebra $A_\C$ gives us coordinates having the
property that the multiplication rule does not mix the two
coordinates: that is, $(a,b)(c,d)=(ac,bd)$.  This non-mixing
property propagates somewhat into $A_\HH$, and we will need to use
this at a key stage below.  Let $\omega\colon A_\HH \ra \A^2$ be the
map $(a_1,a_2,b_1,b_2)\mapsto (a_1,b_2)$.  Then the diagram
\[ \xymatrix{
A_\HH \times A_\HH \ar[r]^\mu \ar[d]_{\id\times \omega} & A_\HH \ar[d]^\omega \\
A_\HH\times \A^2 \ar[r]^-{\mu'} & \A^2
}
\]
commutes,
where $\mu'$ is given by $(a_1,a_2,b_1,b_2)*(x,y)=(a_1x-yb_1,ya_2+b_2x)$.
In words, for all $u$ and $v$ in $A_\HH$,
the first and last coordinates of $uv$ 
only depend on the first and last coordinates of
$v$.  It is somewhat more intuitive to see this using the isomorphism
of Remark \ref{rem:SL2}, 
where the map $\mu'$ corresponds to 
the usual action of matrices on
column vectors (up to some sporadic signs).
\end{remark}

\subsection{The Hopf maps}

The basic definition of the motivic Hopf maps now proceeds just as
in classical homotopy theory.  The reader should review Appendix~\ref{se:Hopf}
at this point, for the definition and properties of the  Hopf
construction.

\begin{defn}
\mbox{}
\begin{enumerate}[(1)]
\item
The first Hopf map $\eta$ is defined to be the element of $\pi_{1,1}(S)$
represented by the  Hopf construction of the multiplication map
$S_\C \times S_\C \map S_\C$.
\item
The second Hopf map $\nu$ is defined to be the element of $\pi_{3,2}(S)$
represented by  the Hopf construction of the multiplication map
$S_\HH \times S_\HH \map S_\HH$.
\item
The third Hopf map $\sigma$ is defined to be the element of $\pi_{7,4}(S)$
represented by  the Hopf construction of the multiplication map
$S_\OO \times S_\OO \map S_\OO$.
\end{enumerate}
\end{defn}

The following result and its proof are  due to Morel \cite[Lemma 6.2.3]{M2}:

\begin{lemma}
\label{lem:epsilon-eta}
The elements $\eta$ and $\eta \epsilon$ are equal in $\pi_{1,1}(S)$.
\end{lemma}

\begin{proof}
Multiplication on $S_\C$ is commutative.  
Recall that $\epsilon$
is represented by the twist map on $S_\C \Smash S_\C$.
The diagram
\[
\xymatrix{
S_\C \Smash S_\C \ar[d]_\epsilon \ar[r]^\chi & 
  S_\C \times S_\C \ar[r]^-\mu \ar[d]_T &
S_\C \ar[d]^= \\
S_\C \Smash S_C \ar[r]_\chi & 
  S_\C \times S_\C \ar[r]_-\mu &
S_\C
}
\]
commutes by Lemma~\ref{lem:Hopf-twist}, where $\mu$ is multiplication
and $T$ is the twist map.  The horizontal compositions represent $\eta$.
\end{proof}

The above lemma deduces an identity involving the Hopf elements as a
consequence of the commutativity of $A_\C$.  The point of the next
section will be to deduce some other identities from deeper properties
of the Cayley-Dickson algebras.

\begin{remark}
The above proof works unstably, but only after three simplicial suspensions.  
As explained in Appendix~\ref{se:stable-split1},
the map $\chi$ exists as an unstable map $\Sigma(S^{1,1}\Smash
S^{1,1})\ra 
\Sigma( S^{1,1}\times S^{1,1})$, which gives a model for $\eta$ in the
unstable group $\pi_{3,2}(S^{2,1})$.  But the left square in the
diagram is only guaranteed to commute after an additional {\it two\/}
suspensions, by Lemma~\ref{lem:Hopf-twist}.  The necessity of some of these
suspensions is demonstrated by the fact that classically one does not
have
$\eta=-\eta$ in $\pi_3(S^2)$.  
\end{remark}

The next result is also due to Morel \cite[Lemma 6.2.3]{M2}.  We
include this result and its proof only for didactic purposes.
The proof demonstrates how one must be careful with orientations,
canonical isomorphisms, and commutativity relations.

\begin{prop}
Let $\P^1$ and $(\A^2-0)$ be oriented as in 
Example~\ref{ex:orient}.
The usual projection
$\pi\colon (\A^2-0) \ra \P^1$ represents the element $\eta$ in
$\pi_{1,1}(S)$.
\end{prop}

\begin{proof}
We have the diagram
\[
\xymatrix{
S_\C \Smash S_\C \ar[r]^\chi \ar[d]_{(-)^{-1} \Smash \id} & 
  S_\C \times S_\C \ar[r]^-f \ar[d]_{(-)^{-1} \times \id} &
  S_\C \ar[d]_= \\
S_\C \Smash S_\C \ar[r]_{\chi} & S_\C \times S_\C \ar[r]_-\mu & S_\C,
}
\]
where $\mu$ is multiplication
and $f$ is the map $(x,y) \mapsto x^{-1}y$.
By definition, $\eta$ is the composition along the bottom of the
diagram.
Recall that the inverse map on $S_\C$ represents $\epsilon$, by
Proposition~\ref{pr:power1}; therefore $(\blank)^{-1}\Smash \id$ represents $\epsilon$ as well, using Remark~\ref{re:invcoh}.
Since $\eta$ equals $\eta \epsilon$ by Lemma~\ref{lem:epsilon-eta},
it suffices to show that the composition along the top of the 
diagram is equivalent to $\pi$, i.e., that $\pi$ is the Hopf
construction on $f$.

Let $U_1 \lla U_1\cap U_2 \lra U_2$ be the standard affine cover of
$\P^1$, as in Example~\ref{ex:orient}(1).  Taking the preimage under
$\pi$ gives a cover of $(\A^2-0)$, so we get the diagram 
\begin{myequation}
\label{eq:2x3}
 \xymatrix{
\pi^{-1}U_1 \ar[d] & \pi^{-1}(U_1\cap U_2)\ar[d]\ar[l]\ar[r] & \pi^{-1}U_2\ar[d] \\
U_1 & U_1\cap U_2\ar[l]\ar[r] & U_2,
}
\end{myequation}
where $(\A^2-0)$ and $\P^1$ are the homotopy pushouts of the 
top and bottom rows respectively.

Diagram (\ref{eq:2x3})
is isomorphic to the diagram
\begin{myequation}
\label{eq:2x3a}
 \xymatrix{
(\A^1-0)\times \A^1 \ar[d]_{(x,y)\mapsto x^{-1}y} & (\A^1-0)\times
(\A^1-0)\ar[r]^-i
\ar[d]^f \ar[l]_-i
& \A^1\times (\A^1-0)\ar[d]^{(x,y)\mapsto xy^{-1}}  \\
\A^1 & (\A^1-0) \ar[r]^{\inv} \ar[l]_i & \A^1
}
\end{myequation}
where all maps labelled $i$ are the inclusions.
This new diagram in turn maps, via a natural weak equivalence, to the diagram
\begin{myequation}
\label{eq:2x3b}
 \xymatrix{
(\A^1-0) \ar[d] & (\A^1-0)\times (\A^1-0)\ar[r]^-{\pi_2}\ar[l]_-{\pi_1}\ar[d]^f & (\A^1-0) \ar[d] \\
{*} & (\A^1-0 )\ar[r]\ar[l] & {*}. 
}
\end{myequation}
Diagram (\ref{eq:2x3b})
induces a map on homotopy pushouts of the rows, which
is equal to the Hopf construction $H(f)$ on $f$ 
by definition (see Appendix~\ref{se:Hopf}). 
So we have produced a zig-zag of equivalences between
$\pi$ and $H(f)$.
The weak equivalences in this zig-zag turn out to be 
orientation-preserving, which follows 
by the definition of our standard orientations in
Example~\ref{ex:orient}.
It follows that $[\pi]=[H(f)]$.    
\end{proof}

\begin{remark}[Nontriviality of $\eta$, $\nu$, and $\sigma$]
It is worth pointing out that if our base $k$ is a
field of characteristic not equal to $2$ then none of $\eta$, $\nu$, and $\sigma$ are
equal to the zero element.
Here it is useful to work unstably: since
the splitting $\chi$ exists after one suspension, $\eta$, $\nu$, and
$\sigma$ can be modelled by unstable maps $S^{3,2}\ra S^{2,1}$,
$S^{7,4}\ra S^{4,2}$, and $S^{15,8}\ra S^{8,4}$.  A completely routine
modification of the standard argument from \cite[Lemma 1.5.3]{SE}
shows that the homotopy cofibers of these maps have nontrivial cup products in
their mod $2$ motivic cohomology---more precisely, the maps have Hopf
invariant one in the usual sense that the square of the generator in
bidegree $(2n,n)$ equals the generator in  dimension $(4n,2n)$, for
$n=1,2,4$ in the three respective cases.   Properties of the motivic
Steenrod squares then show the existence of the expected Steenrod
operations in the cohomology of the cofibers, which proves that the
maps are not stably trivial.

The assumption that the base is a field not of characteristic $2$ is
because it is in that setting that we know the necessary 
results about the Steenrod operations in motivic cohomology.
It of course follows that $\eta$, $\nu$, and $\sigma$ are non-zero over $\Z$ as well.
One can presumably use motivic $\F_3$-cohomology to detect $\nu$ and $\sigma$
over fields of characteristic $2$.
We do not know whether $\eta$ is non-zero over fields of characteristic $2$.

When the base is a field of characteristic zero, another approach is
to reduce to the case $k\inc \C$ and then 
apply the topological realization from motivic homotopy theory to
classical homotopy theory.  The motivic elements $\eta$, $\nu$, and $\sigma$
all map to elements of Hopf invariant one.
\end{remark}


\section{The null-Hopf relation}

The goal of this section is to prove with geometric arguments
that $\eta \nu$ and $\nu \sigma$ are both zero.  The proofs for these
two results follow essentially the same pattern, but in the case of
$\nu\sigma=0$ one part of the argument develops some complications
that require a non-obvious workaround.  Our approach in this section will be
to first concentrate on the $\eta\nu=0$ proof, so that  the reader can
see the basic strategy of what is happening.  Then we repeat most of
the steps for the case of the $\nu\sigma=0$ argument, explaining what
the differences are.

We begin our work by returning to Cayley-Dickson algebras:

\begin{lemma}
\label{lem:t-endo}
Let $A$ be an associative involutive $k$-algebra,
let $\gamma$ be in $k^\times$, and let $t$ be an element
of $A$ having norm $1$.
The map $\theta_t\colon D_\gamma(A)\ra D_\gamma(A)$ given by
$\theta_t(a,b)=(a,t b)$
is an involution-preserving endomorphism of the Cayley-Dickson double
$D_\gamma(A)$.  
In particular, $\theta_t$ is norm-preserving.  
\end{lemma}

\begin{proof}
Verify that $\theta_t$ is an involution-preserving endomorphism
directly with the formulas for $D_\gamma(A)$ given at the beginning
of Section \ref{subsctn:CD-fundamentals}.
Since $n(x)=xx^*$, it follows that $\theta_t$ also preserves the norm.
\end{proof}

From Lemma \ref{lem:t-endo},
we find that there is a pairing 
\[ \theta\colon S(A)\times S(D_\gamma A)\ra S(D_\gamma A)
\]
given by $\theta(t,x)=\theta_t(x)$.  
In particular, this yields maps 
\[ \alpha\colon S_\C \times S_\HH \map S_\HH  \quad\text{and}\quad
 \beta\colon S_\HH \times S_\OO \map S_\OO.
\]
Note that these maps commute with multiplication in the sense that
\[
\alpha(t,ab)  = \alpha(t,a) \alpha(t,b) \qquad\text{and}\qquad
\beta(a,xy)  = \beta(a,x) \beta(a,y).
\]
In other words, the diagram
\begin{myequation}
\label{eq:endo-mult}
\xymatrix@C+1.5ex{
& S_\C \times S_\HH \times S_\HH 
  \ar[r]^-{\Delta \times 1 \times 1} \ar[dd]_{1 \times \mu} &
S_\C \times S_\C \times S_\HH \times S_\HH
  \ar[r]^{1 \times T \times 1} &
S_\C \times S_\HH \times S_\C \times S_\HH
  \ar[d]^{\alpha \times \alpha} \\
& & & S_\HH \times S_\HH \ar[d]^\mu \\
& S_\C \times S_\HH \ar[rr]_\alpha & & 
S_\HH, 
}
\end{myequation}
commutes, where $\Delta$ and $T$ are the evident diagonal and twist maps.
A similar diagram commutes for $S_\HH$, $S_\OO$, and $\beta$.

\begin{lemma}
\label{lem:alpha-eta}
The Hopf construction on $\alpha$ represents $\eta$.
\end{lemma}

\begin{proof}
Recall the orientation-preserving 
weak equivalence $\pi\colon S_\HH \map (\A^2-0)$ that takes 
$(a_1,a_2,b_1,b_2)$ to $(a_1,b_1)$, as well as the isomorphism $p\colon S_\C\ra
(\A^1-0)$ that sends $(t_1,t_2)$ to $t_1$.  We have a commutative diagram
\[
\xymatrix{
S_\C \Smash S_\HH \ar[r]^\chi \ar[d]_{\simeq} &
S_\C \times S_\HH \ar[d]_{\simeq} \ar[r]^\alpha & S_\HH \ar[d]^{\simeq} \\
(\A^1-0) \Smash (\A^2-0) \ar[r]_\chi &
(\A^1-0) \times (\A^2-0) \ar[r]_-{\alpha'} & \A^2-0,
}
\]
where $\alpha'\colon (\A^1-0) \times (\A^2-0) \map (\A^2-0)$ is given by
$(t,(x,y)) \mapsto (x,ty)$.  The diagram shows that 
the Hopf constructions $H(\alpha)$ and $H(\alpha')$ 
represent the same map in $\pi_{*,*}(S)$, so we will now
focus on the latter.

Recall that we have fixed an isomorphism (in the homotopy category) 
between $(\A^2-0)$ and the join $(\A^1-0)*(\A^1-0)$.  Under this
isomorphism, $\alpha'$ coincides with the melding $\pi_2 \# \mu$, 
where $\pi_2$ and $\mu$ are the projection
 and multiplication maps $(\A^1-0)\times (\A^1-0)\ra (\A^1-0)$.
See Appendix~\ref{se:meld} for
the definition of $\pi_2\# \mu$.  

Proposition~\ref{pr:H(meld)} gives a formula for $H(\pi_2\#\mu)$.  But
$H(\pi_2)$ is null by Lemma~\ref{lem:Hopf-trivial}, and so that formula simplifies
to just
\[ [H(\alpha')]=[H(\pi_2\#\mu)]= \tau_{(1,1),(1,1)}\cdot
[H(\mu)]=\epsilon [H(\mu)]=\epsilon \eta=\eta.
\]
The element $\tau_{(1,1),(1,1)}$ is computed by Equation (\ref{eq:tau}), and the
last equality is by Lemma~\ref{lem:epsilon-eta}.
\end{proof}

The next result is the desired null-Hopf relation.

\begin{prop}
\label{prop:eta-nu}
$\eta \nu = 0$.
\end{prop}

\begin{proof}
We will
examine what happens when both routes around Diagram 
(\ref{eq:endo-mult}) are
precomposed with the splitting map
$\chi \colon S_\C \Smash S_\HH \Smash S_\HH \map 
 S_\C \times S_\HH \times S_\HH$.  Note that throughout this proof we
work in the stable category.

We will begin with the lower-left composition.
To analyze $\alpha(1\times \mu)\chi$, use the
commutative diagram
\[
\xymatrix{
S_\C \Smash S_\HH \Smash S_\HH \ar[d]_{1 \Smash \chi} \ar[dr]^{1 \Smash H(\mu)} \\
S_\C \Smash (S_\HH \times S_\HH) \ar[d]_\chi \ar[r]_-{1 \Smash \mu} &
S_\C \Smash S_\HH \ar[d]^{\chi} \ar[dr]^{H(\alpha)} \\
S_\C \times S_\HH \times S_\HH \ar[r]_-{1 \times \mu} &
S_\C \times S_\HH \ar[r]_{\alpha} &
S_\HH.
}
\]
The square commutes because $\chi$ is natural, and 
the two triangles commute by definition of the Hopf construction.  The left vertical
composite equals $\chi$ by Remark~\ref{re:chi-p}.
Recall from Lemma~\ref{lem:alpha-eta} that $H(\alpha)=\eta$, and of
course $H(\mu)=\nu$ by definition.   So we have that
\[ [\alpha(1\times \mu)\chi]=[H(\alpha)]\cdot [1\Smash
H(\mu)]=[H(\alpha)]\cdot [H(\mu)]=\eta\nu.
\]
The second equality uses Remark~\ref{re:invcoh}(i).

Next we analyze what happens when we compose 
$\chi \colon S_\C \Smash S_\HH \Smash S_\HH \map S_\C \times S_\HH \times S_\HH$
with the top-right part of Diagram (\ref{eq:endo-mult}).
 We will obtain zero,
which will finish the proof.
This is mostly an application of Proposition~\ref{pr:H(meld)},
where the maps $f \colon S_\C \times S_\HH \map S_\HH$ and
$g \colon S_\C \times S_\HH \map S_\HH$ are both equal to $\alpha$.

First recall from Corollary~\ref{co:stable-split}(b) that the identity
map on $S_\HH\times S_\HH$ can be written as
$\id_{S_\HH\times S_\HH} = \chi p + j_1\pi_1 + j_2\pi_2$
where $\pi_1$ and $\pi_2$ are the two projections $S_\HH\times
S_\HH\ra S_\HH$; $j_1,j_2\colon S_\HH \ra S_\HH\times S_\HH$ are
the two inclusions as horizontal and vertical slices; and $p$ is the
projection from the product to the smash product.  The composite
of interest can therefore be written as a sum of three composites of the
form
\[ S_\C \Smash S_\HH \Smash S_\HH \llra{\chi}
S_\C \times S_\HH \times S_\HH \llra{h} S_\HH\times S_\HH \llra{u}
S_\HH 
\]
where $h$ denotes the 
composition $S_\C \times S_\HH \times S_\HH \ra S_\HH \times S_\HH$
along the top-right part of Diagram (\ref{eq:endo-mult}),
and $u$ is one of
$\mu\chi p$, $\mu j_1\pi_1=\pi_1$, and $\mu j_2\pi_2=\pi_2$.  
But in the latter two cases the
composites are clearly null; in the case of $\pi_1$, for example,
this follows from the diagram
\[ \xymatrix{
S_\C\Smash S_\HH \Smash S_\HH \ar@{.>}[dr]\ar[r]^-\chi & S_\C\times S_\HH\times
S_\HH \ar[r]^-{\pi_1 h} \ar[d] & S_\HH \\
& S_\C\times S_\HH\times * \ar[ur]_\alpha
}
\]
and the fact that the dotted composite is null by the defining
properties of $\chi$ (Proposition~\ref{pr:chi_stable}).

So it remains to analyze the composite
\[ S_\C \Smash S_\HH \Smash S_\HH \llra{\chi}
S_\C \times S_\HH \times S_\HH \llra{h} S_\HH\times S_\HH \llra{p}
S_\HH\Smash S_\HH \llra{\chi} S_\HH\Smash S_\HH \llra{\mu} S_\HH.
\]
This is equal to $H(\mu)\circ H(\alpha\#\alpha)$, using 
Lemma~\ref{le:meld1} for the second factor.  
Proposition~\ref{pr:H(meld)}  says that
$[H(\alpha\#\alpha)]$ equals
\[
\tau_{(1,1),(3,2)} [H(\alpha)] + \tau_{(1,1),(3,2)}[H(\alpha)] +
(\tau_{(1,1),(3,2)})^2 [H(\alpha)]\cdot [H(\alpha)]\cdot
[\Delta_{S^{1,1}}].
\]
Now use  $[\Delta_{S^{1,1}}]=\rho$ from
Theorem~\ref{th:diagonal}; $[H(\alpha)]=\eta$ from 
Lemma~\ref{lem:alpha-eta};
and $\tau_{(1,1),(3,2)}=1$ by Equation (\ref{eq:tau}).  We obtain
$H(\alpha\# \alpha)= 2\eta+\eta^2\rho$,
which equals zero by
Theorem~\ref{th:Morel}(ii).
\end{proof}

We next duplicate the above arguments to prove the analogous Hopf
relation $\nu\sigma=0$, using the pairing $\beta\colon S_\HH\times
S_\OO\ra S_\OO$.  This time we go through the steps in reverse
order, saving what is now the hardest step for last.

\begin{prop}
\label{prop:nu-sigma}
$\nu \sigma = 0$.
\end{prop}

\begin{proof}
The proof is very similar to the proof of
Proposition~\ref{prop:eta-nu}.  One starts with the 
commutative diagram analogous to Diagram (\ref{eq:endo-mult})
showing that $\beta$ respects multiplication, and
then precomposes the two routes around the diagram with $\chi$.  
By exactly the same
arguments as before, the composition along the bottom-left part of the
diagram gives $[H(\beta)]\cdot \sigma$, and
composition along the top-right part of the diagram gives
\[ \sigma \cdot \Bigl [ \tau_{(3,2),(7,4)}[H(\beta)]+\tau_{(3,2),(7,4)}[H(\beta)] +
\bigl ( \tau_{(3,2),(7,4)} \bigr )^2 [\Delta_{S^{3,2}}] \Bigr ]
\]
(using Proposition~\ref{pr:H(meld)}).  But here the diagonal map is equal to
zero by Theorem~\ref{th:diagonal}, because $S^{3,2}$ is a simplicial
suspension.  Using that
$\tau_{(3,2),(7,4)} = -1$ by Equation $(\ref{eq:tau})$,
our formula becomes
\[
[H(\beta)]\cdot \sigma=\sigma\cdot [-2H(\beta)] = 2[H(\beta)]\cdot \sigma,
\]
We have used graded-commutativity from
Proposition~\ref{pr:commute} in the second equality.
This shows that
$[H(\beta)]\cdot \sigma=0$.  Finally, use that $[H(\beta)]=-\nu$ by
Lemma~\ref{lem:beta-nu}
below.
\end{proof}

Our next goal is to compute the Hopf construction $H(\beta)$.
Recall that the pairing $\beta\colon S_\HH\times
S_\OO\ra S_\OO$ sends $[t,(x,y)]\mapsto (x,ty)$.  The idea is to realize
$S_\OO$ as the join of two copies of $S^{3,2}$, corresponding to the
two coordinates $x$ and $y$.  Under this equivalence, $\beta$ becomes
the melding $\pi_2 \# \mu$ (Section~\ref{se:meld}),
where $\pi_2$ and $\mu$ are the projection and multiplication maps
$S_\HH\times S_\HH\ra S_\HH$.  
Proposition~\ref{pr:H(meld)} then shows
that $[H(\beta)]=\tau_{(3,2),(3,2)}\nu=-\nu$.  

The difficulty comes in realizing $S_\OO$ as a join, in a way that is
compatible with the $\beta$-action by $S_\HH$.  To understand the
problem, it is useful to review how this would work in classical
topology.
Let $S$ be the unit sphere inside the classical octonions $\OO$,
consisting of pairs $(x,y)\in \HH\times \HH$ such that
$|x|^2+|y|^2=1$.  Let $U_1\subseteq S$ be the set of pairs where
$x\neq 0$, and let $U_2\subseteq S$ be the set of pairs where $y\neq
0$.  There are evident projections $q_1\colon U_1 \ra S^3$ and
$q_2 \colon U_2\ra S^3$ given by $q_1(x,y)=\tfrac{x}{|x|}$ and
$q_2(x,y)=\tfrac{y}{|y|}$.  
The diagram
\[ \xymatrix{
U_1 \ar[d]_{q_1}& U_1\cap U_2 \ar[l]\ar[r]\ar[d]^{q_1\times q_2} & U_2\ar[d]^{q_2}
\\
S^3 & S^3\times S^3 \ar[l]_-{\pi_1}\ar[r]^-{\pi_2} & S^3
}
\] 
is commutative; the homotopy colimit of the top row is $S$, and the
homotopy colimit of the bottom row is the join $S^3*S^3$.  
All of the vertical maps are
homotopy equivalences.
Moreover, if we let $S(\HH)=S^3$ act on $S^3\times S^3$ trivially on
the first factor and by left multiplication on the second factor,
then the $S(\HH)$-actions on $S$ and
$S^3\times S^3$ are compatible with respect to the maps in  the above diagram.  
This identifies $S(\HH)\times S\ra S$ with the melding of the two evident
$S(\HH)$-actions on $S^3$.

Unfortunately, the above argument does not work in the motivic
setting.  We do not have square roots, so we cannot normalize
vectors; likewise, the homotopies that show the vertical maps in the
diagram to be equivalences all use square roots.  So the above simple
argument breaks down in several spots.  

We get around these difficulties by using a 
special property of the split quaternions $A_\HH$.  Basically, we use
the splitting to reduce the action to a different model of the same
sphere, where it is easier to see the melding.  

\begin{lemma}
\label{lem:beta-nu}
The Hopf construction on $\beta$ represents $-\nu$.
\end{lemma} 

\begin{proof}
Recall the pairing $\mu'\colon A_\HH \times \A^2 \ra \A^2$ given by
$(a_1,a_2,b_1,b_2)*(x,y)=(a_1x-yb_1,ya_2+b_2x)$
from Remark \ref{re:splitting-H},
as well as the commutative diagram
\[ \xymatrix{ A_\HH \times A_\HH \ar[r]^-\mu \ar[d]_{\id\times \omega} & A_\HH
\ar[d]^\omega \\
A_\HH\times \A^2 \ar[r]^-{\mu'} & \A^2.
}
\]
Note that $\mu'$ restricts to give $S_\HH\times (\A^2-0)\ra
\A^2-0$, and $\omega$ restricts to give an equivalence $S_\HH\ra \A^2-0$.  

Consider now the commutative diagram
\[ \xymatrix{ S_\HH\times S_\OO \ar[r]^-\beta \ar[d]^{\id\times \omega'} 
& S_\OO \ar[d]^{\omega'} \\
S_\HH \times (\A^4-0)\ar[r]^-{\beta'} & \A^4-0
}
\]
where $\omega'(x,y)=(\omega(x),\omega(y))$ and 
$\beta'(t,(u,v)) = (u,t*v)$ for $u$ and $v$ in $\A^2$.
The vertical maps are weak
equivalences by the argument from Example~\ref{ex:orient}(3).
So the Hopf constructions for $\beta$ and $\beta'$
represent the same element of $\pi_{*,*}(S)$.  

We know how to identify the variety $\A^4-0$ as the join
$(\A^2-0)*(\A^2-0)$ (Example~\ref{ex:orient}(2)).
Under this identification, the pairing
$\beta'$ is the melding $\pi_2 \# \mu'$
where $\pi_2$ and $\mu'$ are the projection and multiplication maps
$\S_\HH\times (\A^2-0)\ra (\A^2-0)$. 

Proposition~\ref{pr:H(meld)} now yields the formula 
\[ [H(\beta')]=[H(\pi_2
\# \mu')]=\tau_{(3,2),(3,2)}[H(\mu')]=-[H(\mu')]
\]
using that $H(\pi_2)=0$ from Lemma~\ref{lem:Hopf-trivial}.  Finally,
we turn to the commutative square
\[ \xymatrix{
S_\HH\times S_\HH \ar[r]^-\mu \ar[d]_{\id\times \omega} & S_\HH \ar[d]^{\omega} \\
S_\HH \times (\A^2-0) \ar[r]^-{\mu'} & \A^2-0.
}
\]
The vertical maps are equivalences, so $H(\mu')$ and $H(\mu)$
represent the same element of $\pi_{*,*}(S)$.  
Since $H(\mu)$ is equal to $\nu$ by definition,
we have
\[ 
[H(\beta)]=[H(\beta')]=-[H(\mu')]=-[H(\mu)]=-\nu.
\]
\end{proof}

\begin{remark}
\label{rem:omega-warning}
In the proof of Lemma \ref{lem:beta-nu}, 
we have not established that
$\omega'\colon S_\OO \ra (\A^4-0)$ and $\omega\colon S_\HH \ra (\A^2-0)$
are orientation-preserving.   
By Proposition~\ref{pr:Hopf-or}, this
issue is irrelevant because the homotopy elements represented by 
Hopf constructions are independent of these orientations.  
\end{remark}


\appendix

\section{Stable splittings of products}
\label{se:stable-split}

These appendices develop certain technical homotopy-theoretic
constructions that are used in the body of the paper.
Appendices~\ref{se:stable-split} and \ref{se:joins}, as
well as the first part of Appendix \ref{se:Hopf}, build on ideas that appear 
in the papers of Morel \cite{M1,M2}.  Our aim here is to offer
additional details that are necessary for our proofs.  This
applies, in particular, to the proof of Proposition~\ref{pr:H(meld)},
which is the most important technical tool for
the paper.
These appendices are largely structured with the goal of providing
a comprehensible proof of Proposition \ref{pr:H(meld)}.

\subsection{Generalities}
\label{se:stable-split1}
For most of the applications in this paper, it suffices to work
in the stable category of motivic spectra.  This is also true for the
splittings we are about to discuss, and for the Hopf construction
developed in Section \ref{se:Hopf}.  However, for didactic reasons we are
briefly going to work {\it unstably\/} and be careful about the number
of suspensions required at various stages.  

Let $j\colon A\inc B$ be a cofibration of pointed motivic spaces, and let
$p\colon B\ra B/A$ be the quotient map.  Suppose that there is a map
$\alpha\colon\Sigma B\ra \Sigma A$ that splits $\Sigma j$ in the
homotopy category, i.e., $\alpha(\Sigma j)\he \id_{\Sigma A}$.  For any
pointed object $X$, we have an exact sequence
\[ \cdots \lla [B/A,X]_* \lla [\Sigma A,X]_* \llla{(\Sigma j)^*} [\Sigma B,X]_*
\llla{(\Sigma p)^*} 
[\Sigma(B/A),X]_* \la \cdots
\]
of sets.
Then $\alpha^*$ is a splitting for $(\Sigma j)^*$, 
so $(\Sigma j)^*$ is surjective.
It follows that we have
an exact sequence of groups
\[ 1 
\lla 
[\Sigma A,X]_* \llla{(\Sigma j)^*} [\Sigma B,X]_*
\llla{(\Sigma p)^*} [\Sigma(B/A),X]_* \lla 1.
\]
Because $\alpha$ is not necessarily the suspension of a map $B \ra A$,
the map $\alpha^*$ is not necessarily a group homomorphism.
So the exact sequence is not necessarily split-exact.
Although the groups in
the above sequence need not be abelian, we will still write $+$ for
the group operation and $0$ for the identity. 
When $X$ equals $\Sigma B$, the element
$\id_{\Sigma B} - (\Sigma j)\alpha$ of $[\Sigma B, \Sigma B]_*$
belongs to the kernel of $(\Sigma j)^*$ 
and is therefore in the image of $(\Sigma p)^*$.

\begin{defn}
\label{de:chi}
The map $\chi\colon \Sigma(B/A) \ra \Sigma B$ is the unique
map in the homotopy category of pointed spaces such that 
$\id_{\Sigma B}$ equals $(\Sigma j)\alpha + \chi (\Sigma p)$.
\end{defn}

\begin{lemma}
\label{lem:chi-prop}
The map $\chi$ satisfies:
\begin{enumerate}
\item
$(\Sigma p)\chi=\id_{\Sigma(B/A)}$.
\item
$\alpha \chi=0$.
\end{enumerate}
\end{lemma}

\begin{proof}
Let $X$ be $\Sigma(B/A)$.  Compute that
$(\Sigma p)^* ( (\Sigma p) \chi - \id_{\Sigma B/A} )$ is zero.
Since $(\Sigma p)^*$ is one-to-one, we get that
$(\Sigma p) \chi - \id_{\Sigma B/A}$ is zero.

For the second, let $X$ be $\Sigma A$.
Compute that $(\Sigma p)^* (\alpha \chi)$ is zero.
Since $(\Sigma p)^*$ is one-to-one, we get that
$\alpha \chi$ is zero.
\end{proof}

We now suspend once more to obtain the sequence
\begin{myequation}
\label{eq:2-split}
 0 \lla
[\Sigma^2 A,X]_* \llla{(\Sigma^2 j)^*} [\Sigma^2 B,X]_*
\llla{(\Sigma^2 p)^*} [\Sigma^2(B/A),X]_* \lla 0.
\end{myequation}
This is now a short exact sequence of abelian groups, and it {\it
is\/} split-exact because the map $(\Sigma \alpha)^*$ is a group
homomorphism.  

With at least two suspensions, we have the
following converse to Lemma \ref{lem:chi-prop}.

\begin{lemma}
\label{lem:chi-char}
Let $i \geq 2$.
Suppose $x\colon \Sigma^i (B/A) \ra \Sigma^i B$
is  such that
\begin{enumerate}
\item
$(\Sigma^i p) x$ equals $\id_{\Sigma^i (B/A)}$.
\item
$(\Sigma^{i-1} \alpha) x$ equals zero.
\end{enumerate}
Then $\Sigma x$ equals $\Sigma^{i}\chi$ in
$[\Sigma^{i+1}(B/A),\Sigma^{i+1}B]_*$.  
\end{lemma}

\begin{proof}
We simply compute:
\begin{align*}
\Sigma x =\id\circ(\Sigma x) &= 
[ (\Sigma^{i} j) (\Sigma^{i-1} \alpha) +
(\Sigma^{i-1} \chi)(\Sigma^i p) ] (\Sigma x)\\
& = 
 (\Sigma^{i} j) (\Sigma^{i-1} \alpha)(\Sigma x) +
(\Sigma^{i-1} \chi)(\Sigma^i p) (\Sigma x) \\
& = 
\Sigma^i \chi,
\end{align*}
where the second equality comes from Definition \ref{de:chi} and the
fourth equality comes from the given properties of $x$.  In the third
equality we have used $(A+B)(\Sigma x)=A(\Sigma x)+B(\Sigma x)$; note
that the analogous formula without the suspension does not hold in
general.
\end{proof}

\begin{remark}
We will often apply Lemma \ref{lem:chi-char}
in the case $i=\infty$, where the statement
yields that the stable homotopy class of $\chi$ is characterized by the
given two properties.
\end{remark}

\subsection{Splittings of products}
Now let $X$ and $Y$ be pointed spaces, and specialize to the  cofiber sequence
\begin{myequation}
\label{eq:j-p}
\xymatrix{
X \vee Y \ar[r]^j & X \times Y \ar[r]^p & X \Smash Y.
}
\end{myequation}
Let $\pi_1\colon X \times Y \map X$ and 
$\pi_2\colon X \times Y \map Y$ be the two projection maps.
Let $\alpha$ be the homotopy class
$\Sigma \pi_1 + \Sigma \pi_2\colon \Sigma(X\times Y) \map \Sigma X \vee
\Sigma Y$, defined using the group structure on $[\Sigma(X\times
Y),\Sigma (X\vee Y)]_*$.  
The composition $\alpha (\Sigma j)$ is
the identity (up to homotopy), i.e. $\Sigma \pi_1 + \Sigma \pi_2$ splits $\Sigma j$.

By Definition \ref{de:chi}, we obtain a map
$\chi\colon \Sigma(X \Smash Y) \map \Sigma(X\times Y)$, uniquely
defined up to based homotopy.
This map is a splitting for $\Sigma p$ and
satisfies $(\Sigma \pi_1+\Sigma\pi_2) \chi=0$.
Moreover, Lemma \ref{lem:chi-char} says that 
$\chi$ is completely characterized by these criteria, up to suspension.
When necessary for clarity, we will write
$\chi(X,Y)$ for the map
$\chi\colon \Sigma(X \Smash Y) \map \Sigma(X\times Y)$.

The definition of $\chi$ shows that it is natural in $X$ and $Y$.  That is, if $Z$ and
$W$ are also pointed and $f\colon X\ra Z$ and $g\colon Y\ra W$ are two
based maps, then the diagram
\[ 
\xymatrix{ 
\Sigma(X\Smash Y) \ar[r]^\chi \ar[d]_{\Sigma(f\Smash g)} &
\Sigma(X\times Y) \ar[d]^{\Sigma(f\times g)}\\
\Sigma(Z\Smash W) \ar[r]_\chi & \Sigma(Z\times W)
}
\]
commutes in the based homotopy category.
This is a routine argument, boiling down to the fact that projections
and inclusions are natural.

In several cases we will need to understand the compatibility of
$\chi$ with the twist maps where one interchanges the roles of $X$ and
$Y$:

\begin{lemma}
\label{lem:Hopf-twist}
The diagram
\[
\xymatrix{
\Sigma(X \times Y) \ar[r]^{\Sigma T_\times} & \Sigma(Y \times X) \\
\Sigma(X \Smash Y) \ar[u]^{\chi(X,Y)} \ar[r]_{\Sigma T_\Smash} & \Sigma(Y \Smash X)
\ar[u]_{\chi(Y,X)}
}
\]
commutes (up to homotopy) after two suspensions,
where $T_\times$ and $T_\Smash$ are the evident twist maps.
\end{lemma}

\begin{proof}
Let $f$ denote the composite
\[ \Sigma(X\Smash Y) \llra{\Sigma T_\Smash} \Sigma(Y\Smash
X)\llra{\chi(Y,X)} \Sigma(Y\times X) \llra{\Sigma T_\times} \Sigma(X\times
Y).
\]
We want to show that $\Sigma^2 f$ equals $\Sigma^2 \chi(X,Y)$.
By Lemma \ref{lem:chi-char},
 it suffices to prove that $(\Sigma \pi_1+\Sigma \pi_2)f=0$ and
$(\Sigma p)f=\id_{\Sigma(X\Smash Y)}$.  These follow from
naturality of the twist maps and the corresponding
properties of $\chi(Y,X)$. 
\end{proof}

\begin{remark}
Lemma \ref{lem:Hopf-twist} is almost true after one
suspension, but we need the extra suspension because of the
restriction $i\geq 2$ from Lemma \ref{lem:chi-char}.
\end{remark}

\subsection{Stable considerations}
From now on we 
disregard the suspensions required for the careful
statement of unstable results.
That is, we work in the stable category of spectra.
When $X$ is a pointed space, we will often abuse notation
and write $X$ again for $\Sigma^\infty X$.
Also, it will be convenient to now let $\chi$ denote the desuspension
of the splitting $\Sigma(X\Smash Y)\ra \Sigma(X\times Y)$ produced in
the last section.  So $\chi$ is now a map 
$X\Smash Y\ra
X\times Y$.

Proposition~\ref{pr:chi_stable} below is a stable version of 
Lemma~\ref{lem:chi-char}, with a slight strengenthing due to stability.

\begin{prop}
\label{pr:chi_stable}
The map $\chi$ is the unique stable homotopy class
$X\Smash Y \ra X\times Y$ 
such that 
$ p\chi$ is the identity on $X \Smash Y$ and
$\pi_1\chi=\pi_2\chi=0$ is zero.
\end{prop}

\begin{proof}
Lemma~\ref{lem:chi-char} implies that $\chi$ is the unique stable
homotopy class such that $p\chi=\id_{X\Smash Y}$ and
$(\pi_1+\pi_2)\chi=0$.
It suffices for us to show that the second condition is equivalent to
$\pi_1\chi=\pi_2\chi=0$.  Clearly the latter implies the former, since
$(\pi_1+\pi_2)\chi=\pi_1\chi+\pi_2\chi$ (and note that this uses
stability).  Conversely, if $(\pi_1+\pi_2)\chi=0$ then multiplying by
$\pi_1$ on the left gives $\pi_1\chi=0$ since the composite
$\pi_1\pi_2$ is null.  
Similarly, left multiplication by $\pi_2$ gives $\pi_2\chi=0$.
\end{proof}

\begin{cor}
\label{co:stable-split}
In the stable homotopy category,
\begin{enumerate}[(a)]
\item
$\pi_1+\pi_2+p\colon 
X\times Y \map X \vee Y \vee (X\Smash Y)$
is an isomorphism, and
$j \vee \chi$ is a homotopy inverse.
\item $j(\pi_1+\pi_2) + \chi p$ is the identity on $X \times Y$.
\end{enumerate}
\end{cor}

\begin{proof}
For part (a), use that $(\pi_1 + \pi_2) j = \id_{X \vee Y}$;
$p j = 0$; $(\pi_1 + \pi_2) \chi = 0$; and
$p \chi = \id_{X \Smash Y}$.
For part (b), $j (\pi_1 + \pi_2) + \chi p$ is $\id_{X \times Y}$
by Definition \ref{de:chi}.
\end{proof}

Let $\Delta_\times\colon X \map X\times X$ and
$\Delta_\Smash\colon X \map X \Smash X$
be the evident diagonal maps.  

\begin{lemma}
\label{lem:Hopf-diagonal}
Let $X$ be any pointed motivic space.
Then $\Delta_\times$ and $\chi \Delta_\Smash + j(\pi_1+\pi_2)\Delta_\times$
are equal maps $X \map X \times X$.
\end{lemma}

\begin{proof}
Start with $j(\pi_1+\pi_2) + \chi p = 1$ from Corollary \ref{co:stable-split}, 
and apply $\Delta_\times$
on the right.  Finally,
note that $\Delta_\Smash = p \Delta_\times$.
\end{proof}

Later we will need the following calculation of $\chi$ in 
a specific example.

\begin{lemma}
\label{le:chi-calc}
Let $S^{0,0}$ consist of the two points $1$ and $-1$, where $1$ is the
basepoint.
Let $i$, $j$, and $k$ be the based maps 
$S^{0,0} \map S^{0,0} \times S^{0,0}$ that take $-1$
to $(1,-1)$, $(-1,1)$, and $(-1,-1)$ respectively.
Then 
$\chi\colon S^{0,0} \Smash S^{0,0} \map S^{0,0} \times S^{0,0}$
is stably equal to the composition of the isomorphism
$S^{0,0} \Smash S^{0,0} \llra{\iso} S^{0,0}$
with the map $k - i - j$.
\end{lemma}

\begin{proof}
We apply Proposition~\ref{pr:chi_stable}.  
The result follows from the observations that
\begin{enumerate}
\item
$p i$ and $p j$ are zero, whereas $pk$ is the identity.
\item
$\pi_1 i$ is zero, but $\pi_1 j$ and $\pi_1 k$ are the identity.
\item $\pi_2 j$ is zero, but $\pi_2 i$ and $\pi_2 k$ are the identity.
\end{enumerate}
\end{proof}

\subsection{Higher splittings}

Given based spaces $X_1, \ldots, X_n$ and a subset 
$S = \{ i_1, \ldots, i_k\}$ of $\{1, \ldots, n\}$, write
$X^{\times S}$ for $X_{i_1} \times \cdots \times X_{i_k}$ (where
$i_1<i_2<\cdots<i_k$).
Also, write
$X^{\Smash S}$ for $X_{i_1} \Smash \cdots \Smash X_{i_k}$.
Write $p_S\colon X_1\times\cdots \times X_n \ra X^{\Smash S}$ for the
composition
\[ X_1\times\cdots \times X_n \llra{\pi} X^{\times S} \llra{p} X^{\Smash S},
\]
where $\pi$ is the evident projection.

\begin{prop}
\label{pr:higher-split}
The map $\sum_S p_S\colon X_1\times\cdots \times X_n \ra
\bigWedge_{S} X^{\Smash S}$ is a weak equivalence in the stable
category, where 
the sum and wedge range over all
nonempty subsets $S$ of $\{1,\ldots,n\}$.  
\end{prop}

\begin{proof}
This follows from induction and 
part (a) of Corollary~\ref{co:stable-split}.
\end{proof}

Proposition \ref{pr:higher-split} allows us to make 
the following definition.

\begin{defn}
\label{de:higher-chi}
Let
$\chi_S\colon X^{\Smash S} \ra X_1\times \cdots \times X_n$
be the unique homotopy class of maps
such that:
\begin{enumerate}
\item
$p_S\chi_S=\id$.
\item
$p_T\chi_S=0$ for all $T\neq S$.
\end{enumerate}
\end{defn}

Definition \ref{de:higher-chi}
generalizes the properties of the 2-fold splitting that
$p \chi$ is the identity, while $\pi_1 \chi$ and $\pi_2 \chi$ are both zero.
We will
usually write just $\chi$ instead of $\chi_S$.  Just as for the 
$2$-fold splittings, the maps $\chi_S$ are natural in the objects
$X_1,\ldots,X_n$.  

Analogously to Lemma~\ref{lem:Hopf-twist},
Proposition \ref{pr:higher-join-twist} shows that 
the higher splittings respect permutations of the factors.

\begin{prop}
\label{pr:higher-join-twist}
If $\sigma$ is a permutation of $\{ 1, \ldots, n \}$, then the diagram
\[
\xymatrix{
X_1 \Smash \cdots \Smash X_n \ar[r]\ar[d]_\chi &
X_{\sigma 1} \Smash \cdots \Smash X_{\sigma n} \ar[d]_\chi \\
X_1 \times \cdots \times X_n \ar[r] &
X_{\sigma 1} \times \cdots \times X_{\sigma n}  }
\]
is commutative,
where the horizontal maps are permutations of factors.
\end{prop}

\begin{proof}
The proof is 
similar to the proof of Lemma~\ref{lem:Hopf-twist}.
\end{proof}

Now let $X_1,\ldots,X_n$ be formal symbols, and let 
$w$ be a parenthesized word made form these symbols using the two
operations $\times$ and $\Smash$.  For example, $w$ might be
$(X_1\times X_3)\Smash (X_4\times X_2)$.  Let $w'$ be a word
obtained from $w$ by changing one $\times$ symbol to a $\Smash$
symbol, e.g. $w'=(X_1\Smash X_3)\Smash (X_4\times X_2)$.  We can
regard both $w$ and $w'$ as functors $\Ho(\cC)^n \ra \Ho(\cC)$, and
we let $p$ denote the evident natural transformation $w\ra w'$. 
In our example, $p$ is more precisely
$p_{X_1,X_3}\Smash (\id_{X_4}\times \id_{X_2})$.  
There is also an evident natural
transformation $w'\ra w$ made from maps of the form $\chi_S$, and we denote
this just by $\chi$.  

\begin{remark}
\label{re:chi-p}
One has the following ``coherence results'' for the maps
$\chi$ and $p$:
\begin{enumerate}[(i)]
\item 
Given any two sequences of maps 
\[ w=w_1\llra{\chi} w_2\llra{\chi}\cdots
\llra{\chi} w_r=v
\]
and
\[ w=w_1'\llra{\chi} w_2'\llra{\chi}\cdots
\llra{\chi} w_s=v
\]
with the same source and target,
the two composite natural transformations are equal.
\item  
Given any two sequences of maps 
\[ w=w_1\llra{p} w_2\llra{p}\cdots
\llra{p} w_r=v
\]
and
\[ w=w_1'\llra{p} w_2'\llra{p}\cdots
\llra{p} w_s=v
\]
with the same source and target,
the two composite natural transformations are equal.  
\item 
Let ``$\chi$-map'' now refer to any composition as in (i), and
``$p$-map'' refer to any composition as in (ii).  From now on, if a
map is labelled as $\chi$ or $p$ it means it belongs to one of these
classes.

Let $\alpha$ be a composition as in (i) and let $\beta$ be a
composition as in (ii), and assume that the last word of $\alpha$
coincides with the first word in $\beta$ (so that $\beta\alpha$ makes
sense).  Moreover, assume that the ``spots" which $\beta$ turns
from $\times$ to $\Smash$ form a subset of the ``spots"  which
$\alpha$ turns from $\Smash$ to $\times$.  Then $\beta \alpha$ is a
$\chi$-map.  (This generalizes
the splitting property $p\chi=\id$).
\end{enumerate}
\end{remark}

We will not give proofs for the claims in Remark \ref{re:chi-p},
since proving them in
the stated generality involves an unpleasant amount of bookkeeping.
In all three cases, the proofs boil down to Proposition~\ref{pr:higher-split}.
When we use Remark \ref{re:chi-p} in the context of this paper, it will always be
in cases where only three or four maps are involved.  In those
cases, it is easy enough to check the claims by hand.  

\begin{ex}
Here are some examples to demonstrate the use of Remark~\ref{re:chi-p}.
\begin{enumerate}[(1)]
\item
The compositions
\[
\xymatrixrowsep{1pc}\xymatrixcolsep{3.5pc}\xymatrix{
X \Smash Y \Smash Z \ar[r]^{1 \Smash \chi(Y,Z)} &
X \Smash (Y \times Z) \ar[r]^-{\chi(X,Y\times Z)} & 
  X \times Y \times Z \\
X \Smash Y \Smash Z \ar[r]^{\chi(X,Y) \Smash 1} &
  (X \times Y) \Smash Z \ar[r]^-{\chi(X\times Y,Z)} & 
  X \times Y \times Z }
\]
are both equal to $\chi(X,Y,Z)$ in the homotopy category.
\item
The composition
\[
\xymatrixcolsep{4.9pc}\xymatrix{
W \Smash X \Smash Y \Smash Z \ar[r]^-{\chi(W,X) \Smash \chi(Y,Z)} &
(W \!\times\! X) \Smash (Y\! \times \! Z) \ar[r]^-{\chi(W\!\times \!X,Y\!\times\! Z)} &
W \!\times\! X \!\times\! Y\! \times\! Z  
}
\]
is  equal to $\chi(W,X,Y,Z)$ in the homotopy category.
\item
The triangle
\[ \xymatrix{X\Smash Y\Smash Z\ar[r]^\chi\ar[dr]_{\chi(X,Y)\Smash
1} & X\times Y\times Z\ar[d]^p \\
& (X\times Y)\Smash Z
}
\]
commutes for any objects $X$, $Y$, and $Z$.  
\end{enumerate}
\end{ex}


\section{Joins and other homotopically canonical constructions}
\label{se:joins}

In the next two sections we will need to deal with several
homotopical constructions.  In each situation, the output is not just a
single object but rather a whole contractible category of objects.
Things become complicated when we want to identify
the outputs of different multi-layered constructions as being essentially the same.
We start with a brief review of the machinery needed to handle
these kinds of situations.

\subsection{Canonical constructions}
\label{se:canon}

Suppose that $\cM$ is a model category, $I$ is a small category, and
$X\colon I\ra \cM$ is a diagram.  Homotopy theorists are faced with
the troublesome fact that there is not a single homotopy colimit for
$X$; rather, there are many different models for the homotopy colimit,
but they are all weakly equivalent to each other.  The trouble really begins
when one needs to {\it choose\/} a weak equivalence between two
different models, and then use this to perform further constructions.
One cannot choose the weak equivalence arbitrarily and expect to
obtain consistent results later on.   

As an elementary example, suppose 
that $A$ and $B$ are models for $\hocolim_I X$, and choose a weak
equivalence $A\ra B$.
If at some later stage one similarly chooses  a weak
equivalence $B\ra A$, then the composition $A\ra B\ra A$ may or may not be
homotopic to the identity.

One way to control these issues is as follows.  In good cases,
one can define a model category structure on the category of diagrams
$\cM^I$, where the weak equivalences and fibrations are the objectwise
ones \cite{H}.  Let $\Cof(X)$ be the category of cofibrant
approximations to $X$ in this model structure.  This category is
contractible \cite{H}, and the colimit of any object in $\Cof(X)$
gives a model for $\hocolim_I X$.  The image of the composition
\[ \Cof(X) \llra{\colim} \cM \ra \Ho(\cM)
\]
is a contractible groupoid; so for any two objects $A$ and $B$ in the
image, there is a unique isomorphism $A\ra B$ that is also in the
image.  Note that there might be many different isomorphisms in
$\Ho(\cM)$ from $A$ to $B$, but only {\it one\/} of them lies in the image of
the above composite.    In this sense there is a ``homotopically
canonical'' isomorphism between $A$ and $B$.

The considerations of the previous paragraph give a solution to our
problem, but it is not a simple one.  Given two models $A$ and $B$
for $\hocolim_I X$, we only can get our hands on the canonical
homotopy equivalence between them by finding diagrams $D$
and $D'$ in $\Cof(X)$ and specifying $A$ and $B$ as the colimit of
these diagrams---in essence, one must specify {\it why\/} $A$ and $B$
are models for $\hocolim_I X$, and only then does one get the
comparison map.  

A simple example demonstrates what is happening here.  The suspension of
a topological space $X$ is a homotopy colimit for the diagram $*\lla X \lra *$.
Given two spaces $A$ and $B$ that happen to have the homotopy type of
$\Sigma X$, there is not a unique way to get a weak equivalence $A\ra B$,
even up to homotopy.  However, if one specifies a decomposition 
into ``top" and ``bottom" cones for both $A$ and $B$, then one {\it
does\/} obtain a comparison map that is unique up to homotopy. 
The choice of top and bottom cones gives
a diagram $[C_+ \lla X \lra C_-]$ that is a cofibrant model for $*\lla{X}\lra *$.  

Now suppose that $I$ and $J$ are two small categories, and let
$X\colon I\ra \cM$ and $Y\colon J\ra \cM$ be two diagrams.  We think of
$\hocolim_I X$ as a contractible groupoid inside of $\Ho(\cM)$, and
likewise for $\hocolim_J Y$. 
Let $A$ and $B$ be specific models
for these two homotopy colimits, and let $A\ra B$ be a map.
For any other models $\hat{A}$ and $\hat{B}$ for the two homotopy colimits,
we immediately obtain corresponding maps $\hat{A}\ra
\hat{B}$ in $\Ho(\cM)$.
Namely, we have the composite 
$\hat{A}\iso A \lra B\iso \hat{B}$, where the first
and last isomorphisms are the ones from the respective contractible
groupoids.
In this case, we say that the maps $A\ra B$ and $\hat{A}\ra
\hat{B}$ are ``canonically isomorphic''.  

Often in practice, we have a certain {\it procedure\/} for
producing a map between models for $\hocolim_I X$ and $\hocolim_J Y$.
We want to know that this procedure, applied to $A$ and $B$, or
applied to $\hat{A}$ and $\hat{B}$, gives maps that are canonically
isomorphic---i.e., conjugate to each other via the contractible
groupoids for $\hocolim_I X$ and $\hocolim_J Y$.  This is often the
case, but it is something that needs to be {\it checked\/}; it is not
automatic.  

Unfortunately, there is no known efficient and
carefree way to keep track of all of these kinds of compatibilities.
While a direct approach works in simple arguments,
this becomes harder in multi-layered constructions.
We will see
some examples below.  
One precise but cumbersome technique is to work always at the level of diagrams,
i.e., work in $\cM^I$ and related categories as much as possible, rather
than work in $\Ho(\cM)$.  We follow this approach in most of our
arguments.  

When considering suspensions in a model category, we let $I$ be the
pushout category $0\lla 1 \lra 2$.  The category $\cM^I$ has a model
structure where the weak equivalences are objectwise and where the
cofibrant objects are diagrams 
\[
X_0 \lla X_1 \lra X_2
\]
such that each
$X_i$ is cofibrant and both maps are cofibrations.

\begin{defn}
Let $X$ be any object of $\cM$.
\dfn{Suspension data} for $X$ is a
cofibrant diagram $C_+ \lla QX \lra C_-$ where $C_+$ and $C_-$ are
contractible, together with a weak equivalence $QX\ra X$.
\end{defn}

See also Remark \ref{rem:susp-data} for a discussion of suspension data.
This is the same as specifying a cofibrant replacement for $*\lla X
\lra *$ in $\cM^I$.
Every collection of suspension data gives rise to a model for the
suspension of $X$, namely the pushout $C_+\amalg_{QX} C_-$.

\medskip

 In many places in mathematics one learns how
to handle technical details and then immediately starts to leave them
in the background, seemingly ignored.   
Our discussion of canonical
constructions is one of these instances.  While there are real issues
that require attention whenever such constructions appear, giving
{\it complete\/} details in proofs quickly becomes an obstruction
rather than an aid to comprehension; such
details are best left to the reader.  Certain canonical equivalences
will be ubiquitous throughout the rest of these appendices---a clear
example is in the statement of Lemma~\ref{le:meld1}, but most often the
equivalences are appearing with less acknowledgement.  The present
section was meant to provide a kind of ``global'' acknowledgement that
this is what is going on.

\subsection{Joins}
\label{se:join1}
\mbox{}

We now construct the join of two objects and establish a connection
between the join and the splitting maps $\chi$ from Appendix~\ref{se:stable-split}.

\begin{defn}
Given objects $X$ and $Y$, the join $X*Y$ is the homotopy
colimit of the diagram $X\lla X\times Y \lra Y$.
\end{defn}

Note that the diagram
\[ \xymatrix{
X \ar[d] & X\times Y\ar[d]\ar[l]\ar[r] & Y\ar[d] \\
{*} & X\times Y\ar[l]\ar[r] & {*}
}
\]
yields a canonical map $X*Y\llra{\gamma} \Sigma(X\times Y)$. 

\begin{lemma}
\label{lem:join-we}
The composite 
\[
X*Y\llra{\gamma} \Sigma(X\times Y) \llra{\Sigma p} \Sigma(X\Smash Y)
\]
is a weak equivalence.
\end{lemma}

\begin{proof}
Let $X \cof CX$ and $Y \cof CY$
be cofibrations with contractible target. 
Consider the diagram
\[ \xymatrix{
 X \Wedge CY \ar@{ >->}[d] & X\Wedge Y \ar@{ >->}[r]\ar@{ >->}[d]\ar@{ >->}[l] &
CX\Wedge Y \ar@{ >->}[d] \\
X\times CY \ar[d] & X\times Y \ar[d] \ar@{ >->}[l] \ar@{ >->}[r] &
CX\times Y\ar[d]
\\
X\Smash CY & X\Smash Y \ar@{ >->}[r]\ar@{ >->}[l]& CX\Smash Y.
}
\]
The pushout of  each row is a model for the homotopy pushout, because
of the horizontal cofibrations.  The pushout of the top row is
$CX\Wedge CY$; the pushout of the middle row is a model for $X*Y$, and
the pushout of the last row is a model for $\Sigma(X\Smash Y)$ because
both $X\Smash CY$ and $CX\Smash Y$ are contractible. 
 
The columns of the diagram are homotopy cofiber sequences, so taking
homotopy pushouts of each row gives a new homotopy cofiber sequence
\[
CX\Wedge CY \cof X*Y \ra \Sigma(X\Smash Y).
\]
But $CX\Wedge CY$ is contractible, so the second map is a weak
equivalence.  

There is an evident weak equivalence between the two diagrams
\[ 
\xymatrixrowsep{1pc}\xymatrix{X\times CY\ar[d] & X\times
Y\ar[r]\ar[l]\ar[d] 
& CX\times Y \ar[d]\\
X\Smash CY & X\Smash Y\ar[r]\ar[l] & CX\Smash Y
}
\qquad\qquad 
\xymatrix{
X\ar[d] & X\times Y\ar[d]\ar[r]\ar[l] & Y\ar[d]\\
{*} & X\Smash  Y \ar[r]\ar[l] & {*}
}
\]
mapping the diagram on the left to the one on the right.
On taking homotopy pushouts of the rows, the first diagram
gives the weak equivalence of the previous paragraph,
while the second diagram gives $(\Sigma p) \gamma$.
\end{proof}

Proposition \ref{pr:chi-model} below
shows that $\gamma$ is a model for $\chi$, 
once we identify $X*Y$ with $\Sigma(X\Smash Y)$.

\begin{prop}
\label{pr:chi-model}
The map $\chi$ is equal to
the composition
\[
 \Sigma(X\Smash Y) \llra{\simeq} X*Y \llra{\gamma} \Sigma(X\times Y),
\]
where the first map is the homotopy inverse to $(\Sigma p)\gamma$.
\end{prop}

\begin{proof}
Let $\chi'$ be the composition under consideration.
By Proposition~\ref{pr:chi_stable}, it suffices
to show that $(\Sigma p)\chi'$ is the identity on $\Sigma(X\Smash Y)$
and that $(\Sigma \pi_1 + \Sigma \pi_2)\chi'$ is zero.
The first of these is immediate from the definition of $\chi'$.
For the second, observe that $(\Sigma \pi_1)\gamma$ and
$(\Sigma \pi_2)\gamma$ are both zero.  For example, in the first case this
follows from the diagram
\[
\xymatrixrowsep{1pc}\xymatrix{
X\ar[d] & X\times Y\ar[r]\ar[l]\ar[d] & Y\ar[d] \\
X\ar[d] & X \ar[r]\ar[l]\ar[d] & {*}\ar[d] \\
{*} & X \ar[r]\ar[l] & {*}.
}
\]
Upon taking homotopy colimits of the rows,
the diagram induces $(\Sigma \pi_1)\gamma$, but the homotopy
colimit of the middle row is contractible.
\end{proof}

We will also need the following simple result.

\begin{prop}
\label{pr:join-hocolim-commute}
Let $D\colon I\ra \sPre(\Sm/k)$ be a diagram of motivic spaces, and
let $X$ be a fixed motivic space.  Then there is a canonical
equivalence between $\hocolim_I [X*D_i]$ and $X*(\hocolim_I D)$.  
\end{prop}

\begin{proof}
Let $J$ be the pushout indexing category $1\lla 0 \lra 2$ and let
$\overline{D}\colon I\times J\ra \sPre(\Sm/k)$ be the evident diagram
where
\[ \overline{D}(i,1)=X, \quad \overline{D}(i,0)=D(i)\times X, \quad
\overline{D}(i,2)=D(i)
\]
(the maps in $\overline{D}$ are the obvious ones).  
A
standard result in the theory of homotopy colimits gives canonical
equivalences between $\hocolim \overline{D}$, $\hocolim_I [\hocolim_J
D]$, and $\hocolim_J [\hocolim_I D]$, where the latter two expressions
indicate the evident iterated homotopy colimits along slices of
$I\times J$.  Now just observe that $\hocolim_I[\hocolim_J
D]=\hocolim_I (X*D_i)$ and $\hocolim_J[\hocolim_I D]$ is canonically
equivalent to $X*[\hocolim_I D]$; the latter uses that homotopy
colimits commute with products by a fixed space, which is a standard
property of simplicial presheaf categories (following from the analogous result
for $\sSet$).    
\end{proof}


\section{The Hopf construction}
\label{se:Hopf}

We now describe and study the Hopf construction.

\begin{defn}
\label{de:Hopf}
Let $X$, $Y$, and $Z$ be pointed spaces, and let $h\colon
X\times Y \ra Z$ be a pointed map.  The \dfn{Hopf construction} of $h$
is the map $H(h)\colon X*Y\ra \Sigma Z$ obtained by taking
homotopy colimits of the rows of the diagram
\[ 
\xymatrixrowsep{1pc}\xymatrix{
X\ar[d] & X\times Y\ar[r]\ar[l]\ar[d] & Y\ar[d] \\
{*} & Z \ar[r]\ar[l] & {*}.
}
\] 
\end{defn}

We often regard $H(h)$ as a map $\Sigma(X\Smash Y)\ra \Sigma Z$
(or as just a map $X\Smash Y\ra Z$) using the standard equivalence
$X*Y\he \Sigma(X\Smash Y)$ from Lemma \ref{lem:join-we}.

Lemma~\ref{lem:Hopf-trivial} below is a simple example of the Hopf construction.

\begin{lemma}
\label{lem:Hopf-trivial}
Let $\pi_1 \colon X \times Y \map X$ 
and $\pi_2 \colon X \times Y \map Y$ 
be the evident projection maps.  The maps
$H(\pi_1)$ and $H(\pi_2)$ are trivial.
\end{lemma}

\begin{proof}
Consider the diagram from the end of the proof of Proposition \ref{pr:chi-model}.
The map $H(\pi_1)$ is obtained by taking the homotopy colimits of the rows,
but the homotopy colimit of the middle row is contractible.

The argument for $H(\pi_2)$ is identical.
\end{proof}

Using the model for $\chi$ given in Proposition~\ref{pr:chi-model}, we
can give an alternative  model for the Hopf construction on $h\colon X\times Y \ra Z$:

\begin{prop} 
\label{prop:Hopf-model}
Let $X$, $Y$, and $Z$ be pointed spaces, and let $h\colon
X\times Y \ra Z$ be a pointed map.  
The Hopf construction $H(h)$ equals the composite
\[ X*Y\llra{\he} \Sigma(X\Smash Y) \llra{\chi} \Sigma(X\times Y)
\llra{\Sigma h} \Sigma Z,
\]
where the first map is the weak equivalence from Lemma \ref{lem:join-we}.
\end{prop}

\begin{proof}
The diagram
\[ \xymatrixrowsep{1pc}\xymatrix{
X\ar[d] & X\times Y\ar[r]\ar[l]\ar[d] & Y\ar[d] \\
{*}\ar[d] & X\times Y \ar[r]\ar[l]\ar[d]^h & {*}\ar[d] \\
{*} & Z \ar[r]\ar[l] & {*}.
}
\]
shows that $H(h)$ is equal to $(\Sigma h)\gamma$, where $\gamma$
is from Section~\ref{se:join1}.  By
Proposition~\ref{pr:chi-model}, the composition $X*Y\llra{\he} \Sigma(X\Smash
Y)\llra{\chi} \Sigma(X\times Y)$ is equal to $\gamma$.
\end{proof}

The following simple result will save us a bit of
trouble in the body of the paper.

\begin{prop}
\label{pr:Hopf-or}
Let $f\colon
X\times Y\ra Y$ be a pointed map, where $X$ and $Y$ are oriented
homotopy spheres.  Then $[H(f)]$ does not depend on
the orientation of $Y$.
\end{prop}

\begin{proof}
Let $\sigma\colon Y\ra Y$ be an automorphism in the homotopy category,
and consider the diagram
\[ \xymatrix{
X\Smash Y\ar[d]^{1\Smash \sigma}\ar[r]^\chi & X\times Y
\ar[d]^{1\times \sigma}
\ar[r]^-f & Y\ar[d]^\sigma
\\
X\Smash Y\ar[r]^\chi &  X\times Y\ar[r]^-g & Y
}
\]
where $g=\sigma f(1 \times \sigma^{-1})$.  We must show that
$[H(f)]=[H(g)]$.  But the diagram yields
\[ [\sigma]\cdot [H(f)]=[H(g)]\cdot [1\Smash \sigma]=[H(g)]\cdot
[\sigma]
\]
using Remark~\ref{re:invcoh} in the last step.  Since
$[\sigma]$ is central and invertible, $[H(f)]=[H(g)]$.  
\end{proof}

For the following proposition, regard $S^{0,0}$ as the group scheme $\Z/2$.
This result is not  needed in the present paper, but we include it
for future reference.  

\begin{prop}
The Hopf construction on the product map 
$\mu\colon S^{0,0}\times S^{0,0} \ra S^{0,0}$
represents the element $-2$ in $\pi_{0,0}(S)$.  
\end{prop}

\begin{proof}
The composition
\[ 
\Sigma S^{0,0}\llra{\iso} \Sigma(S^{0,0}\Smash S^{0,0}) \llra{\chi}
\Sigma(S^{0,0}\times S^{0,0}) \llra{\Sigma \mu} \Sigma S^{0,0}
\]
is the Hopf construction on $\mu$, which is equal to 
$(\Sigma\mu) (\Sigma k-\Sigma i-\Sigma j)$, where $i$, $j$, and $k$
are as in Lemma~\ref{le:chi-calc}.  This is evidently the same as
$\Sigma(\mu k)-\Sigma(\mu i)-\Sigma (\mu j)$.  But $\mu k$ is null,
and both $\mu i$ and $\mu j$ are the identity maps.  
It follows that $(\Sigma \mu)(\Sigma k-\Sigma i-\Sigma j)$ equals $-2$.
\end{proof}

\subsection{Hopf constructions of meldings}
\label{se:meld}

Let $f_1\colon X\times Y_1\ra Z_1$ and
$f_2\colon X\times Y_2\ra Z_2$ be two based maps, 
where $X$, $Y_1$, $Y_2$, $Z_1$, and $Z_2$ are pointed spaces.
Consider the diagram
\begin{myequation}
\label{eq:meld}
\xymatrixrowsep{1.3pc}\xymatrix{
X\times Y_1\ar[d]_{f_1} & 
  X\times Y_1\times Y_2\ar[d]^{\alpha(f_1,f_2)} \ar[r]\ar[l] 
  & X\times Y_2 \ar[d]^{f_2} \\
Z_1 & Z_1\times Z_2 \ar[r]\ar[l] & Z_2,
}
\end{myequation}
where the horizontal maps are the evident projections and
$\alpha(f_1,f_2)$ is the composite
\[ \xymatrixcolsep{2.5pc}\xymatrix{
X\times Y_1\times Y_2 \ar[r]^-{\Delta\times 1\times 1} 
& X\times X\times Y_1\times Y_2 \ar[r]^{1\times T \times 1} & X\times
Y_1\times X\times Y_2 \ar[r]^-{f_1\times f_2} & Z_1\times Z_2.
}
\]

\begin{defn}
\label{de:melding}
The \dfn{melding} $f_1 \# f_2$ of the pairings $f_1$ and $f_2$ is the
pairing
\[
X \times (Y_1 * Y_2) \ra \Sigma (Z_1 * Z_2)
\]
obtained by
taking homotopy pushouts of the rows of Diagram (\ref{eq:meld}).
\end{defn}

Lemma \ref{le:meld1} below gives a tool for computing Hopf constructions of
meldings.

\begin{lemma}
\label{le:meld1}
The Hopf construction $H(f_1\# f_2)\colon X*(Y_1 * Y_2)\ra
\Sigma^2 (Z_1* Z_2)$ is canonically identified with the double suspension
of the composite
\[ X\Smash Y_1\Smash Y_2 \llra{\chi} X\times Y_1\times Y_2
\llra{\alpha(f_1,f_2)} Z_1\times Z_2 \llra{p} Z_1\Smash Z_2.
\]
\end{lemma}

\begin{proof}
This is an exercise in the manipulation of homotopy colimits.  
Consider the rectangular box
\[ 
\xymatrixrowsep{2pc}
\xymatrixcolsep{3pc}
\xymatrix@!0{
X \ar[dr] && X\times Y_1\ar[dr]\ar[ll]\ar[rr] && Y_1\ar[dr] \\
& {*} && Z_1\ar[ll]\ar[rr] && {*} \\
X \ar[dr]\ar[uu]\ar[dd] && X\!\times \!(Y_1\!\!\times\!\!
Y_2)\ar[dr]\ar'[l][ll]
   \ar'[r][rr]\ar'[u][uu]\ar'[d][dd] && Y_1\times Y_2 \ar'[u][uu]\ar'[d][dd]\ar[dr] \\
& {*}\ar[uu]\ar[dd] && Z_1\times Z_2\ar[uu]\ar[dd] \ar[ll]\ar[rr]&&
  {*}\ar[uu]\ar[dd] \\
X \ar[dr]&& X\times Y_2\ar'[l][ll]\ar'[r][rr]\ar[dr] && Y_2 \ar[dr] \\
& {*} && Z_2\ar[ll]\ar[rr] && {*}.
}
\]
All of the horizontal and vertical maps are the evident projections,
while the three maps in the middle column coming out of the page
are $f_1$, $\alpha(f_1,f_2)$, and $f_2$.

Let $P_f$ be the front $3 \times 3$ diagram, and let
$P_b$ be the back $3 \times 3$ diagram.
The whole diagram is a  natural transformation $P_b\ra P_f$.  
We will compute the induced map  $g: \hocolim P_b \ra \hocolim P_f$
in two different ways.

We can calculate the homotopy colimit of a $3\times 3$ grid in two ways:
first take homotopy pushouts of the rows, and then take the homotopy pushout
of the resulting column; or first take the homotopy pushouts of the columns,
and then take the homotopy pushout of the resulting row.  If we first take
homotopy colimits of the columns, then we obtain the diagram
\[ 
\xymatrixrowsep{3pc}
\xymatrixcolsep{3pc}
\xymatrix@!0{
X\ar[rd] 
& & X\times (Y_1* Y_2) \ar[dr]^(0.6){f_1 \# f_2}\ar[ll] \ar[rr] & & Y_1*Y_2 \ar[dr] \\
& {*} & & Z_1* Z_2 \ar[ll]\ar[rr]  & & {*}.
}
\]
Now take homotopy colimits along the rows to yield $H(f_1\# f_2)$.  

On the other hand, if we first take homotopy colimits along the rows, then we obtain
the left face of the diagram
\[
\xymatrixrowsep{0.8pc}
\xymatrixcolsep{0.5pc}
\xymatrix{
X*Y_1 \ar[rr]\ar[dr]_{H(f_1)} & & {*} \ar[dr] \\
& \Sigma Z_1 \ar[rr] & & {*} \\
X * (Y_1 \times Y_2) \ar'[r][rr]\ar[uu]\ar[dd]\ar[dr]_{H(\alpha)} 
& & X * (Y_1 \times Y_2) \ar'[u][uu]\ar'[d][dd]\ar[dr] \\
&  \Sigma (Z_1 \times Z_2) \ar[rr]\ar[uu]\ar[dd] & & 
  \Sigma (Z_1 \times Z_2) \ar[uu]\ar[dd] \\
X * Y_2 \ar'[r][rr]\ar[dr]_{H(f_2)} & & {*} \ar[dr] \\
& \Sigma Z_2 \ar[rr] & & {*},
}
\]
where $\alpha$ is $\alpha(f_1,f_2)$.
Taking the homotopy pushouts of the columns and applying the evident canonical
isomorphisms, we get the commutative diagram
\[
\xymatrixcolsep{0.5pc}
\xymatrix{
\Sigma^2 X \Smash Y_1 \Smash Y_2 \ar[dr]_{\tilde{g}}
\ar[rr]^{\Sigma^2\bigl (1 \Smash \chi\bigr )} 
& & \Sigma^2 X \Smash (Y_1 \times Y_2) \ar[dr]^{\Sigma^2 H(\alpha)} \\
& \Sigma^2 Z_1 \Smash Z_2 \ar[rr]_\chi & & \Sigma^2 (Z_1 \times Z_2).
}
\]
Note that 
Proposition~\ref{pr:join-hocolim-commute} has been used for one of the
columns.  
Here  $\tilde{g}$ is canonically equivalent to the map 
$g\colon \hocolim P_b \ra \hocolim P_f$   that we are trying to understand,
and the horizontal maps have been identified with the help of
Proposition \ref{pr:chi-model}.

Ignoring suspensions for simplicity,
the above parallelogram allows us to write
\[
\tilde{g} = p \chi \tilde{g} = p \circ H(\alpha)\circ (1 \Smash
\chi)=p \alpha \chi (1\Smash \chi)=p\alpha\chi.
\]
We have used the symbol $\chi$ for slightly different things;
Remark~\ref{re:chi-p} makes sense of this.

In the end, the maps
$H(f_1\# f_2)$, $g$, and $\tilde{g}=p\alpha\chi$ are canonically identified.
\end{proof}

Before stating the next result we introduce a piece of notation.  If
$X$ is an oriented homotopy sphere then write $|X|$ for the unique
pair $(p,q)\in \Z^2$ such that $X\he S^{p,q}$ (uniqueness follows, for
example, by base-extending to a field and then using motivic
cohomology calculations).  Recall that if $X$ and $Y$ are
oriented homotopy spheres then the twist map $T\colon X\Smash Y
\ra Y\Smash X$ represents an element $\tau_{|X|,|Y|}=[T]$ in
$\pi_{0,0}(S)$.   As discussed in 
Section~\ref{se:signs}, if
$X\he S^{p,q}$ and $Y\he S^{s,t}$ then
$\tau_{|X|,|Y|}=\tau_{(p,q),(s,t)}=(-1)^{(p-q)(s-t)}\cdot \epsilon^{qt}$.
These elements are central in $\pi_{*,*}(S)$ because every element of 
$\pi_{0,0}(S)$ is central.
 
In analyses that involve extensive sign calculations, it is very
convenient to drop the absolute value signs and  write 
$\tau_{X,Y}$ for $\tau_{|X|,|Y|}$.  In fact we carry this to an
extreme: if the name of a homotopy sphere appears inside a subscript
for a $\tau$-expression, it is to be interpreted as the associated
bidegree.
For example, 
$\tau_{X+Y-Z,W}$ is shorthand for $\tau_{|X|+|Y|-|Z|,|W|}$.  In
practice this never leads to any confusion.  Since
$\tau_{(\blank),(\blank)}$ is bilinear and takes values in the
$2$-torsion subgroup of the multiplicative group $\pi_{0,0}(S)^\times$, 
we can also write formulas like
\[
\tau_{X+Y-Z,W}=\tau_{X,W}\tau_{Y,W}\tau_{Z,W}^{-1}\tau_{X,W}=
\tau_{X,W}\tau_{Y,W}\tau_{Z,W}\tau_{X,W}.
\]

The following proposition gives a key formula used in the paper.  The
complexity of the signs is unfortunate, but there seems to be no
avoiding this. 
The lack of symmetry in the signs on the first two terms
 is tied to the asymmetry in the signs for $[\id_{r,s}\Smash
f]$ and $[f\Smash \id_{r,s}]$ appearing in Remark~\ref{re:invcoh}. 

\begin{prop}
\label{pr:H(meld)}
Suppose given oriented homotopy spheres $X$, $Y_1$, $Y_2$, $Z_1$, and
$Z_2$, together with pointed maps $f\colon X\times Y_1\ra Z_1$ and
$g\colon X\times Y_2\ra Z_2$.   
Let $f^*$ and $g^*$ denote the composites
\[ Y_1\iso *\times Y_1 \lra X\times Y_1 \llra{f} Z_1 \qquad\text{and}\qquad
Y_2\iso *\times Y_2 \lra X\times Y_2 \llra{g} Z_2.
\]
Then $[H(f \# g)]$ equals
\begin{align*}
 \tau_{X+Y_1-Z_1,Z_2}[H(f)]\cdot [g^*] +
&\tau_{X,Y_1}\tau_{Y_1-Z_1,Z_2}[f^*]\cdot [H(g)] \\
& + \tau_{X,Y_2}\tau_{X+Y_1-Z_1,Z_2}
[H(f)]\cdot [H(g)] \cdot [\Delta_X].
\end{align*}
\end{prop}

\begin{proof}
By Lemma~\ref{le:meld1} the map $H(f \# g)$ can be modelled by the 
composite
\[ \xymatrixcolsep{2.9pc}\xymatrix{
X\Smash Y_1\Smash Y_2 \ar[r]^\chi & X\!\times\! Y_1\!\times\!
Y_2 \ar[r]^-{\Delta_\times\times 1} & X\!\times\! X\!\times\! Y_1\!\times\! Y_2
\ar[r]^{1\times T \times 1} &
X\!\times\! Y_1\!\times\! X\!\times\! Y_2 \ar[d]^{f \times g}
\\
&&  Z_1\Smash Z_2 & Z_1\times Z_2\ar[l]_p.
}
\]
We use the fact that $\Delta_\times \times 1
=(j_1\times  1) + (j_2\times 1 )
 + (\chi\Delta_\Smash \times 1)$
from Lemma \ref{lem:Hopf-diagonal}.
So our composite is the sum of
three pieces, which we analyze separately.  

\vspace{0.1in}
\noindent{\mdfn{The $j_1$-composite}}.  This piece is the 
composition along the top right in the diagram
\[ \xymatrix{
X\Smash Y_1\Smash Y_2 \ar[dr]_{\chi\Smash \id}\ar[r]^\chi & X\times Y_1\times Y_2\ar[d]^p
\ar[r]^-{f \times g^*} & Z_1\times Z_2 \ar[d]^p  \\
& (X\times Y_1)\Smash Y_2 \ar[r]_-{f\Smash g^*} & Z_1\Smash Z_2.
}
\]
The composite along the bottom is $H(f)\Smash g^*$.  
Remark~\ref{re:invcoh}(iii) yields the formula 
$[H(f)\Smash g^*]=[H(f)]\cdot [g^*]\cdot
\tau_{X+Y_1-Z_1,Z_2}$.  

\vspace{0.1in}
\noindent{\mdfn{The $j_2$-composite}}.  This piece is the 
composition along the top right in the diagram
\[ \xymatrix{
X\Smash Y_1\Smash Y_2\ar[d]_{T\Smash 1}\ar[r]^\chi & X\times Y_1\times
Y_2\ar[r]^-{T\times 1} &
Y_1\times X\times Y_2\ar[r]^-{f^*\times g}\ar[d]^p & Z_1\times Z_2 \ar[d]^p \\
Y_1\Smash X\Smash Y_2 \ar[urr]^\chi \ar[rr]^{1 \Smash \chi} && Y_1\Smash (X\times
Y_2)\ar[r]^-{f^*\Smash g}
& Z_1\Smash Z_2.
}
\]
The upper left region  commutes by
Proposition~\ref{pr:higher-join-twist}, 
and the middle region 
commutes by Remark~\ref{re:chi-p}.  The composition along the bottom left
of the diagram is
$(f^*\Smash H(g)) (T \Smash 1)$.  Remark~\ref{re:invcoh}(iii)
yields the formula 
$[f^*\Smash H(g)] [T \Smash 1]=
[f^*]\cdot [H(g)]\cdot \tau_{Y_1-Z_1,Z_2} \cdot \tau_{X,Y_1}$.

\vspace{0.1in}
\noindent{\mdfn{The $\chi\Delta_\Smash$-composite}}.  Here we examine
the diagram
\[ \xymatrix{
X\Smash Y_1\Smash Y_2 \ar[r]^\chi\ar[d]_{\Delta_\Smash \Smash 1} &
X\times Y_1\times Y_2\ar[d]^{\Delta_\Smash \times 1} \\
(X\Smash X)\Smash Y_1\Smash Y_2 \ar[d]_{1\Smash T\Smash 1}
\ar[dr]^\chi\ar[r]^\chi & (X\Smash X)\times Y_1\times
Y_2\ar[d]^\chi \\
X\Smash Y_1\Smash X\Smash Y_2\ar[d]_{\chi\Smash \chi}\ar[dr]^\chi & X\times X\times Y_1\times
Y_2 \ar[d]^{1\times T\times 1}\\
(X\times Y_1)\Smash (X\times Y_2) \ar[d]_{f\Smash g} & X\times Y_1\times X\times
Y_2 \ar[l]^-p\ar[d]^{f \times g} \\
Z_1\Smash Z_2 & Z_1\times Z_2.\ar[l]^p
}
\]
The parallelogram in the center commutes by
Proposition~\ref{pr:higher-join-twist}, and the adjacent triangles
commute by Remark~\ref{re:chi-p}.  The composition 
along the top, right, and bottom 
is the composite 
in which we are interested,
whereas the composition along the left is
$[H(f)\Smash H(g)]\cdot \tau_{X,Y_1} \cdot [\Delta_\Smash \Smash 1]$.
By Remark~\ref{re:invcoh}(ii), we have
\[ 
[\Delta_\Smash \Smash 1] = 
[\Delta_\Smash]\cdot
\tau_{-X,Y_1 + Y_2}=
[\Delta_\Smash]\cdot
\tau_{X,Y_1}\tau_{X,Y_2},
\]
and by Remark~\ref{re:invcoh}(iii) we have
\[ [H(f)\Smash H(g)]=[H(f)]\cdot [H(g)]\cdot \tau_{X+Y_1-Z_1,Z_2}.
\]
Putting everything together now gives that our
$\chi\Delta_\Smash$-composite equals
\[ [H(f)]\cdot [H(g)]\cdot \tau_{X+Y_1-Z_1,Z_2}\cdot \tau_{X,Y_1}\cdot
[\Delta_{\Smash}]\cdot \tau_{X,Y_1}\tau_{X,Y_2}.
\]
Note that the two $\tau_{X,Y_1}$ terms cancel, leading to the
expression given in the statement of the proposition.
\end{proof}


\bibliographystyle{amsalpha}

\end{document}